\documentclass[a4paper,12pt]{article}
\usepackage[utf8]{inputenc}
\usepackage[english]{babel}
\usepackage{sectsty}
\usepackage{mathtools,amssymb, amsthm, amsmath, amsfonts, stmaryrd, eufrak,mathpazo, mathrsfs}
\usepackage[backend=biber]{biblatex}
\usepackage{graphicx, subfig}
\usepackage{hhline}
\usepackage{rotating}
\usepackage{geometry}
\usepackage{tasks}
\usepackage{titlesec}
\usepackage[absolute]{textpos}
\usepackage{tikz, tikz-qtree, tikz-cd}
\tikzset{
    lablv/.style={anchor=south, rotate=90, inner sep=.5mm},
    labld/.style={anchor=south, rotate=150, inner sep=.8mm},
    lablad/.style={anchor=north, rotate=40, inner sep=1.2mm}
}
\usepackage{pgffor,etoolbox}
\usepackage{pgfplotstable}
\usepackage{longtable}
\usepackage{booktabs}
\pgfplotsset{compat=1.18}
\usepackage{fp}
\usepackage{xstring}
\usepackage{comment}
\usepackage{youngtab ,ytableau}
\usepackage{enumitem}
\usepackage{fancyhdr}
\usepackage{colortbl}
\usepackage{color}
\definecolor{valentia}{rgb}{233,78,82}
\definecolor{titleblue}{rgb}{114, 146, 162}
\usepackage{array}
\usepackage{makecell}
\usepackage{float}
\usepackage{phoenician}
\usepackage{mdframed}
\usepackage{hyperref}
\usepackage{upgreek}
\usepackage{authblk}
\usepackage{multirow}
\usepackage{multicol}
\DeclareFontFamily{U}{BOONDOX-calo}{\skewchar\font=45 }
\DeclareFontShape{U}{BOONDOX-calo}{m}{n}{
  <-> s*[1.05] BOONDOX-r-calo}{}
\DeclareFontShape{U}{BOONDOX-calo}{b}{n}{
  <-> s*[1.05] BOONDOX-b-calo}{}
\DeclareMathAlphabet{\mathcalb}{U}{BOONDOX-calo}{m}{n}
\SetMathAlphabet{\mathcalb}{bold}{U}{BOONDOX-calo}{b}{n}
\DeclareMathAlphabet{\mathbcalb}{U}{BOONDOX-calo}{b}{n}
\usepackage{scrextend}
\DeclareMathAlphabet{\mathscr}{U}{rsfso}{m}{n}
\DeclareFontFamily{OT1}{pzc}{}
\DeclareFontShape{OT1}{pzc}{m}{it}{ <-> s*[1.2] pzcmi7t }{}
\DeclareMathAlphabet{\mathpzc}{OT1}{pzc}{m}{it}
\DeclareMathSymbol{\smin}{\mathbin}{AMSa}{"39}
\usepackage{authblk}
\newcommand{\Addresses}{{
  \bigskip
  \footnotesize

  R.~Paegelow, \textsc{Cité scientifique, 59655 Villeneuve-d'Ascq FRANCE.}\par\nopagebreak
  \textit{E-mail address}, R.~Paegelow: \texttt{raphael.paegelow@univ-lille.fr}
  \medskip

}}
\title{Fixed points in Gieseker spaces and blocks of Ariki-Koike algebras}

\author{Raphaël Paegelow}
\begin{document}
\newgeometry{top=2cm, bottom=2.5cm, left=2.5cm, right=2.5cm}
\theoremstyle{definition}
\newtheorem{deff}{Definition}[section]
\newtheorem{nota}[deff]{Notation}

\newtheoremstyle{cremark}
    {\dimexpr\topsep/2\relax}
    {\dimexpr\topsep}
    {}
    {}
    {}
    {.}
    {.5em}
    {}

\theoremstyle{cremark}
\newtheorem{rmq}[deff]{Remark}
\newtheorem{ex}[deff]{Example}

\theoremstyle{plain}

\newtheorem{thm}[deff]{Theorem}
\newtheorem{cor}[deff]{Corollary}
\newtheorem{prop}[deff]{Proposition}
\newtheorem{lemme}[deff]{Lemma}
\newtheorem{fait}[deff]{Fact}
\newtheorem{conj}[deff]{Conjecture}
\newtheorem*{theorem1}{Theorem}
\newtheorem*{proposition1}{Proposition a}
\newtheorem*{proposition2}{Proposition b}
\newtheorem*{lemme1}{Lemma}
\newtheorem*{theoremi}{Theorem A}
\newtheorem*{theoremii}{Theorem B}
\newtheorem*{theoremiii}{Theorem C}
\newtheorem*{cor1}{Corollary}


\newcommand{\TODO}{\colorbox{red}{\textbf{{TODO}}}}

\newcommand{\DMQ}{\overline{Q^0_{\Gamma}}}
\newcommand{\DMQGf}{\overline{Q_{G^f}}}
\newcommand{\DMQf}{\overline{Q_{\Gamma^f}}}
\newcommand{\DMQJf}{\overline{Q_{\bullet^f}}}

\newcommand{\QVG}{\mathcal{M}^{\Gamma}}
\newcommand{\tQVG}{\tilde{\mathcal{M}}^{\Gamma}}
\newcommand{\QV}{\mathcal{M}}
\newcommand{\Gie}{\mathpzc{G}}
\newcommand{\tQV}{\tilde{\mathcal{M}}}
\newcommand{\QVTM}{\mathcal{M}^{\Gamma}_{\boldsymbol{\theta}}(M)}
\newcommand{\QVT}{\mathcal{M}^{\Gamma}_{\boldsymbol{\theta}}}
\newcommand{\QVTdef}{\mathcal{M}^{\Gamma}_{\boldsymbol{\theta},\lb}}
\newcommand{\tQVTdef}{\tilde{\mathcal{M}}^{\Gamma}_{\boldsymbol{\theta}, \lb}}
\newcommand{\tQVT}{\tilde{\mathcal{M}}^{\Gamma}_{\boldsymbol{\theta}}}
\newcommand{\tQVTM}{\tilde{\mathcal{M}}^{\Gamma}_{\boldsymbol{\theta}}(M)}
\newcommand{\QVTMs}{\mathcal{M}^{\Gamma}_{\boldsymbol{\theta}}(M)}
\newcommand{\tQVTMs}{\tilde{\mathcal{M}}^{\Gamma}_{\boldsymbol{\theta}}(M)}
\newcommand{\tQVTMb}{\tilde{\mathcal{M}}^{\Gamma}_{\boldsymbol{\theta},\lb\delta^{\Gamma}}(M)}
\newcommand{\QVTMb}{\mathcal{M}^{\Gamma}_{\boldsymbol{\theta},\lb\delta^{\Gamma}}(M)}
\newcommand{\QVTMdef}{\mathcal{M}^{\Gamma}_{\lb\delta^{\Gamma}}(M)}
\newcommand{\tQVTMdef}{\tilde{\mathcal{M}}^{\Gamma}_{\lb\delta^{\Gamma}}(M)}
\newcommand{\QVTd}{\mathcal{M}^{\Gamma}_{\boldsymbol{\theta}}(d)}
\newcommand{\tQVTd}{\tilde{\mathcal{M}}^{\Gamma}_{\boldsymbol{\theta}}(d)}
\newcommand{\QVJn}{\mathcal{M}^{\bullet}_{\boldsymbol{1}}(n)}
\newcommand{\tQVJn}{\tilde{\mathcal{M}}^{\bullet}_{\boldsymbol{1}}(n)}
\newcommand{\QVJnG}{{\mathcal{M}^{\bullet}_{\boldsymbol{1}}(n)}^{\Gamma}}
\newcommand{\QVJndefu}{{\mathcal{M}^{\bullet}_{1}(n)}}
\newcommand{\tQVJndefu}{\tilde{\mathcal{M}}^{\bullet}_{1}(n)}
\newcommand{\QVJnGdef}{{\mathcal{M}^{\bullet}_{1}(n)}^{\Gamma}}
\newcommand{\tQVJndef}{\tilde{\mathcal{M}}^{\bullet}_{\lb}(n)}
\newcommand{\tQVJns}{\tilde{\mathcal{M}}^{\bullet}_{\boldsymbol{\theta}}(n)}
\newcommand{\QVJns}{{\mathcal{M}^{\bullet}_{\boldsymbol{\theta}}(n)}}
\newcommand{\QVJndef}{{\mathcal{M}^{\bullet}_{\lb}(n)}}
\newcommand{\tQVJnb}{\tilde{\mathcal{M}}^{\bullet}_{\boldsymbol{\theta},\lb}(n)}
\newcommand{\QVJnb}{{\mathcal{M}^{\bullet}_{\boldsymbol{\theta},\lb}(n)}}

\newcommand{\Hom}{\mathrm{Hom}}
\newcommand{\GL}{\mathrm{GL}}
\newcommand{\Tr}{\mathrm{Tr}}
\newcommand{\Aut}{\mathrm{Aut}}
\newcommand{\mult}{\mathrm{mult}}
\newcommand{\module}{\mathrm{mod}}
\newcommand{\Dyn}{\mathrm{Dyn}}
\newcommand{\End}{\mathrm{End}}
\newcommand{\SL}{\mathrm{SL}}
\newcommand{\SU}{\mathrm{SU}}
\newcommand{\G}{\mathrm{G}}
\newcommand{\op}{\mathrm{op}}
\newcommand{\RepGa}{\mathbf{McK}(\Gamma)}
\newcommand{\Fr}{\mathrm{Fr}}
\newcommand{\Reg}{\mathrm{Reg}}
\newcommand{\id}{\mathrm{id}}
\newcommand{\Res}{\mathrm{Res}}
\newcommand{\Ind}{\mathrm{Ind}}
\newcommand{\std}{\mathrm{std}}
\newcommand{\Repm}{\mathbf{Rep}(\overline{Q_{\Gamma^f}})}

\newcommand{\Hk}{\mathcal{H}}
\newcommand{\Ck}{\mathscr{X}}
\newcommand{\Xk}{\mathcal{X}}

\newcommand{\Rep}{\mathrm{Rep}}
\newcommand{\RepG}{\mathrm{Rep}_{\Gamma}}
\newcommand{\RepQ}{\mathcal{R}}
\newcommand{\RepGn}{\mathrm{Rep}_{\Gamma,n}}
\newcommand{\RepGk}{\mathrm{Rep}_{\Gamma,k}}
\newcommand{\CharG}{\mathrm{Char}_{\Gamma}}
\newcommand{\CharGn}{\mathrm{Char}_{\Gamma,n}}
\newcommand{\CharGk}{\mathrm{Char}_{\Gamma,k}}
\newcommand{\Xstd}{X_{\mathrm{std}}}

\newcommand{\BD}{BD}
\newcommand{\Gamman}{\Gamma_n}
\newcommand{\Dt}{\tilde{D}}
\newcommand{\LDt}{L(\Dt_{l+2})}
\newcommand{\Dtc}{\tilde{D}}
\newcommand{\Ct}{\tilde{C}}
\newcommand{\Dtl}{\tilde{D}_{l+2}}
\newcommand{\WaD}{W({\tilde{D}_{l+2}})}
\newcommand{\WD}{W({D_{l+2})}}
\newcommand{\IG}{\mathrm{Irr}_{\Gamma}}

\newcommand{\Qc}{Q^{\vee}}
\newcommand{\QDts}{Q(\Dt_{l+2}^{\sigma})[0^+]}
\newcommand{\deltD}{\delta^{\widetilde{BD}_{2l}}}
\newcommand{\av}{\alpha^{\vee}}
\newcommand{\ABDz}{\mathcal{A}^{n,\T_1}_{\BD_{\ell}}}

\newcommand{\Wa}{W_{\Gamma}^{\text{aff}}}

\newcommand{\Taut}{\mathcal{T}_{d}}
\newcommand{\TautT}{\tilde{\mathcal{T}}_{d}}

\newcommand{\bigzero}{\mbox{\normalfont\Large\bfseries 0}}
\newcommand{\rvline}{\hspace*{-\arraycolsep}\vline\hspace*{-\arraycolsep}}
\newcommand{\C}{\mathbb{C}}
\newcommand{\Z}{\mathbb{Z}}
\newcommand{\R}{\mathbb{R}}
\newcommand{\Q}{\mathbb{Q}}
\newcommand{\lb}{\lambda}
\newcommand{\lbb}{\lambda^{\bullet}}
\newcommand{\mub}{\mu^{\bullet}}
\newcommand{\gammab}{\gamma^{\bullet}}
\newcommand{\iso}{\xrightarrow{\,\smash{\raisebox{-0.65ex}{\ensuremath{\scriptstyle\sim}}}\,}}
\newcommand{\Sk}{\mathfrak{S}}
\newcommand{\Uk}{\mathcal{U}}
\newcommand{\T}{\mathbb{T}}
\newcommand{\Pro}{\mathscr{P}}
\newcommand{\Resol}{\mathscr{R}}
\newcommand{\Pnlo}{\mathcal{P}_{n,\ell}^{o}}
\newcommand{\IC}{\mathcal{IC}}
\newcommand{\gl}{\mathrm{g}_{\ell}}
\newcommand{\gG}{\mathrm{g}_{\Gamma}}
\newcommand{\gk}{\mathrm{g}_k}
\newcommand{\gll}{\mathrm{g}_{2l}}
\newcommand{\Cln}{C_{\ell,n}}
\newcommand{\Ctwo}{{(\C^2)}}
\newcommand{\oando}{\boldsymbol{\theta},\lb}
\newcommand{\Ch}{\mathfrak{C}}
\newcommand{\Alf}{\mathfrak{A}}
\newcommand{\ssf}{\mathcalb{s}}
\newcommand{\tkk}{\mathfrak{t}}
\newcommand{\Ssf}{\mathcalb{S}}
\newcommand{\ttf}{\mathcalb{t}}
\newcommand{\CP}{\mathcal{CP}}

\newcommand{\qphoe}{\textphnc{q}}
\newcommand{\Sphoe}{\textphnc{\As}}
\newcommand{\hethphoe}{\textphnc{\Ahd}}


\DeclareRobustCommand*{\Sagelogo}{
  \begin{tikzpicture}[line width=.2ex,line cap=round,rounded corners=.01ex,baseline=-.2ex]
    \draw(0,0) -- (.75em,0)
      -- (.75em,.7ex) -- (.25em,.7ex)
      -- (.25em,.75\ht\strutbox) -- (.75em,.75\ht\strutbox);
    \draw(2em,0) -- (1.6em,0)
      -- (1.3em,.75\ht\strutbox) -- (.9em,.75\ht\strutbox)
      -- (.9em,0) -- (1.3em,0)
      -- (1.6em,.75\ht\strutbox) -- (2.1em,.75\ht\strutbox)
      -- (2.1em,-\dp\strutbox) -- (1.45em,-\dp\strutbox);
    \draw(3em,0) -- (2.25em,0)
      -- (2.25em,.75\ht\strutbox) -- (2.8em,.75\ht\strutbox)
      -- (2.8em,.7ex) -- (2.35em, .7ex);
  \end{tikzpicture}
}
\maketitle


\section*{Abstract}
In this article, we establish combinatorial links between the irreducible components of the fixed point locus of the Gieseker variety and the block theory of Ariki-Koike algebras. First, we describe the fixed point locus in terms of Nakajima quiver varieties over the McKay quiver of type A. We then reinterpret the dimension of an irreducible component as double the weight of a block. Cores of charged multipartitions have been defined by Fayers and further developed by Jacon and Lecouvey. In addition, we give a new way to compute the multicharge associated with the core of a charged multipartition. Finally, we also explain how the notion of core blocks, defined by Fayers, is interpreted on the geometric side using the deep connection between quiver varieties and affine Lie algebras.

\section*{Introduction}
Ariki-Koike algebras, introduced by Susumu Ariki and Kazuhiko Koike \cite{AK94}, is a class of Hecke algebras associated with the complex reflection group ${(\Z/r\Z)^n \rtimes \Sk_n}$ generalising the Hecke algebra for the symmetric group. The representation theory of Ariki-Koike algebras has deep connections with the representation theory of quantum groups conjectured in \cite{LLT96} and proven in \cite{A96}, rational Cherednik algebras \cite[section 5]{GGOR} and cyclotomic Khovanov-Lauda-Rouquier algebras \cite{R08, BK09}. One of the central features of the representation theory of Ariki-Koike algebras which has been heavily developed is its block theory \cite{Fay07, J17, Fay19, Ros19, JL19, Lyle22, DJ24}. A block is an indecomposable two-sided ideal. It plays a crucial role in understanding the modular representation theory of these algebras. A strong motivation is the computation of decomposition matrices. Being cellular algebras, Ariki-Koike algebras have a special class of representations referred to as cellular \cite[section 5]{GL96}. The decomposition matrices store the isotypic multiplicities of the cellular modules (which are all irreducible in characteristic zero) called decomposition numbers. Note that to each block, we can associate a decomposition matrix which is a block of the decomposition matrix of the entire algebra. Thus, the decomposition matrix is a block diagonal matrix which reduces the computation to the blocks. Decomposition matrices present such a great interest because they entirely control the representation theory of Ariki-Koike algebras.\\
This article is the first step towards a geometric approach to the block theory of Ariki-Koike algebras. We propose to link it with the irreducible components of the fixed point locus of the Gieseker spaces. This is the family of symplectic algebraic varieties constructed as Nakajima quiver varieties over the Jordan quiver. It can be interpreted as a generalisation of the punctual Hilbert scheme in the plane.\\
\indent In the first section, we will define the Gieseker space denoted by $\Gie(n,r)$, for each pair of integers $(n,r)$. This quiver variety is naturally equipped with an action of ${\GL_2(\C) \times \GL_r(\C)}$. Let $\mathcal{T}$ be the product of the diagonal torus of $\SL_2(\C)$ and of the diagonal torus of $\GL_r(\C)$. If $\ell$ is a given integer and $\ssf \in ({\Z/\ell\Z})^r$, consider $\Gamma_{\ssf} < \mathcal{T}$ the cyclic group of order $\ell$, where the part in the diagonal torus of $\GL_r(\C)$ is embedded using $\ssf$.  In the second section, we describe the irreducible components of $\Gie(n,r)^{\Gamma_{\ssf}}$ in terms of quiver varieties over the McKay quiver of the cyclic group. We prove the following theorem.

\begin{theoremi}
\label{decomp}
For all non-zero integers $n$, $r$ and $\ell$ and for each $\ssf \in \left(\Z/\ell\Z\right)^r$, the scheme of $\Gamma_{\ssf}$-fixed points in the Gieseker space decomposes into irreducible components
\begin{equation*}
\Gie(n,r)^{\Gamma_{\ssf}} = \coprod_{d \in \mathcal{A}^n_{\Gamma, \ssf}}{\QV_d^{\Gamma_{\ssf}}}
\end{equation*}
\noindent where $\mathcal{A}^n_{\Gamma, \ssf} \subset \{ \text{characters of $n$-dimensional representations of } \Gamma_{\ssf}\}$ and $\QV_d^{\Gamma_{\ssf}}$ is isomorphic to the quiver variety over the McKay quiver with dimension parameter $d$.
\end{theoremi}

In the third section we give another description of the indexing set $\mathcal{A}^n_{\Gamma, \ssf}$ in terms of elements of the root lattice of the affine Lie algebra of type $A_{\ell}$ and recall some combinatorics on multipartitions and abaci.\\
\indent In the fourth section, we recall the block theory of Ariki-Koike algebras. Nakajima and Yoshioka \cite{NakYo03} proved that the points of $\Gie(n,r)^\mathcal{T}$ are parametrized by $r$-multipartitions of size $n$. For $\lbb$ an $r$-multipartition of size $n$, denote by $I_{\lbb}$ the associated $\mathcal{T}$-fixed point.\\
\indent The fifth section contains the main results that link the irreducible components of the scheme $\Gie(n,r)^{\Gamma_{\ssf}}$ with the block theory of Ariki-Koike algebras.

\begin{proposition1}
For each $r$-multipartition $\lbb$, the dimension of the irreducible component containing $I_{\lbb}$ is double the weight of the block of the Ariki-Koike algebra containing the Specht module indexed by $\lbb$.
\end{proposition1}

\begin{proposition2}
Two points $I_{\lbb},I_{\mub} \in \Gie(n,r)^{\mathcal{T}}$ are in the same $\Gamma_{\ssf}$-irreducible component if and only if $\lbb$ and $\mub$ are in the same block of the Ariki-Koike algebra $H_{n,r}(\ell,\ssf)$.
\end{proposition2}

\noindent In \cite{Fay19}, the author has generalised the notion of $\ell$-core to charged multipartitions and the combinatorial structure of this notion has been made explicit in \cite{JL19}. The $\ell$-core of an $\ell$-charged $r$-multipartition $(\lbb,\ssf)$ is another $\ell$-charged $r$-multipartition $(\gammab,\ssf')$. In general, $\ssf'$ is not equal to $\ssf$. One can associate a maximal dominant weight $\Lambda^+_{\lbb,\ssf}$ of an irreducible representation of the affine Lie algebra of type $A_{\ell}$ to $(\lbb,\ssf)$. We will prove that $\ssf'$ can be easily computed using the irreducible components of $\Gie(n,r)^{\Gamma_{\ssf}}$.
\begin{theoremii}
For each $r$-multipartition $\lbb$ and $\ell$-multicharge $\ssf \in {(\Z/\ell\Z)}^r$,
\begin{equation*}
\Lambda^{\ssf'} \equiv \Lambda^+_{\lbb,\ssf} \quad [\delta^{\Gamma}]
\end{equation*}
\noindent where $(\gammab,\ssf')$ is the $\ell$-core of $(\lbb,\ssf)$.
\end{theoremii}

\noindent Finally, Fayers \cite{Fay07} has introduced the notion of core blocks of the Ariki-Koike algebra which can be thought of as an analogue of simple blocks of the Iwahori–Hecke algebra. We prove the following characterisation of such blocks.
\begin{theoremiii}
Let $B$ be a block of $H_{n,r}(\ell,\ssf)$. The following are equivalent.
\begin{itemize}
\item $B$ is a core block.
\item For any $r$-multipartition $\lbb$ in $B$, the $(\ell,\ssf)$-residue is a maximal weight of the irreducible representation of the affine Lie algebra of type $A_{\ell}$ associated with $\ssf$.
\end{itemize}
\end{theoremiii}

\section*{Acknowledgments}
The author would like to thank Gwyn Bellamy for suggesting the generalisation to Gieseker spaces and Nicolas Jacon for many fruitful discussions. Finally, the author would also like to thank Marien Hanot for his moral support and stimulating exchanges.

\section{The Gieseker space}

The Jordan quiver, denoted by $Q_{\bullet}$, is the following quiver
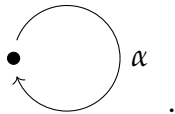
\begin{figure}[H]
  \centering
  \begin{tikzpicture}
    \draw[fill = black] (-0.7 ,0) circle (0.08);
    \draw ([shift=(160 : 0.7)]0, 0) arc (160 : 0: 0.7) node [anchor = west] {$\alpha$};
    \draw[->] (0.7, 0)  arc (0 : -160 :  0.7);
  \end{tikzpicture}
  \label{jordan quiver}.
\end{figure}
\noindent Let $\overline{Q_{\bullet^f}}$ be the double framed quiver of $Q_{\bullet}$
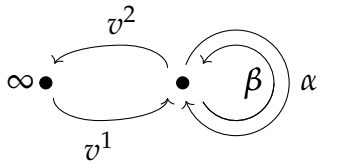
\begin{figure}[H]
  \centering
  \begin{tikzpicture}
    \draw[fill = black] (-0.7 ,0) circle (0.08);
    \draw[fill = black] (-2.5 ,0) circle (0.08) node [anchor = east] {$\infty$};
    \draw ([shift=(160 : 0.7)]0, 0) arc (160 : 0: 0.7) node [anchor = west] {$\alpha$};
    \draw ([shift=(-150 : 0.5)]0, 0) arc (-150 : 0: 0.5) node [anchor = east] {$\beta$};
    \draw (-0.9, 0.2) arc (0 : 90: 0.6 and 0.3 ) node [anchor = south] {$v^2$};
    \draw (-2.4, -0.2) arc (-180 : -90: 0.6 and 0.3 ) node [anchor = north] {$v^1$};
    \draw[->] ([shift=(90:0.5)]-1.5, 0)  arc (90 : 170 :  0.9 and 0.3);
    \draw[->] ([shift=(-90:0.5)]-1.8, 0)  arc (-90 : 0 :  0.9 and 0.3);
    \draw ([shift=(-150 : 0.5)]0, 0) arc (-150 : 0: 0.5) node [anchor = east] {$\beta$};
    \draw[->] (0.5, 0)  arc (0 : 150 :  0.5);
    \draw[->] (0.7, 0)  arc (0 : -160 :  0.7);
  \end{tikzpicture}
  \label{double jordan quiver}.
\end{figure}
\noindent Fix two integers $n$ and $r$ greater or equal than $1$.\\
Let
\begin{equation*}
\RepQ^{\bullet}(n,r):=M_n(\C)^2 \oplus M_{(n,r)}(\C) \oplus M_{(r,n)}(\C)
\end{equation*}
be the representation space of the Jordan quiver with fixed dimensions $n$ and $r$ for the framing. This vector space is equipped with the following $\GL_n(\C)$-action. For all $(g,(\alpha,\beta,v^1,v^2)) \in \GL_n(\C) \times \RepQ^{\bullet}(n,r)$,
\begin{equation}
\label{GLn act}
g.(\alpha,\beta,v^1,v^2):= (g\alpha g^{-1}, g\beta g^{-1}, gv^1, v^2g^{-1}).
\end{equation}
This $\GL_n(\C)$-action is Hamiltonian and the momentum map is given by
\begin{center}
$\mu_{n,r}:\begin{array}{ccc}
\RepQ^{\bullet}(n,r) & \to & M_n(\C)\\
(\alpha,\beta,v^1,v^2) & \mapsto & [\alpha,\beta] + v^1v^2 \\
\end{array}$.
\end{center}

\noindent  If $\theta \in \mathbb{Z}$, define $\chi_{\theta}:\GL_n(\C) \to \C^{\times}, g \mapsto \mathrm{det}(g)^{\theta}$. For $\theta \in \mathbb{Q}$, there exists $N$ such that $N\theta \in \mathbb{Z}$. Denote by $\tilde{\Gie}_{\theta}(n,r):={\mu_{n,r}^{-1}(0)}^{\theta-ss}$ the $\chi_{N\theta}$ semistable points of $\mu_{n,r}^{-1}(0)$. The Nakajima quiver variety attached to the double framed Jordan quiver with stability parameter $\theta$ will be denoted by $\Gie_{\theta}(n,r) \simeq \tilde{\Gie}_{\theta}(n,r) \sslash \GL_n(\C)$.

\begin{rmq}
Note that for $\theta \neq 0$, $\Gie_{\theta}(n,1)$ is isomorphic to the Hilbert scheme of $n$ points in $\C^2$. Thanks to \cite[Lemma $10.29$ \& Theorem $11.18$]{Kir06}, $\Gie_{\theta}(n,r) \simeq \Gie_{\theta'}(n,r)$, whenever $(\theta,\theta') \in {(\Q^{\times})}^2$ .
\end{rmq}

\begin{deff}
The Gieseker space is the algebraic variety $\Gie(n,r):=\Gie_1(n,r)$. Denote moreover $\tilde{\Gie}_{1}(n,r)$ by $\tilde{\Gie}(n,r)$.
\end{deff}

\noindent One has the following lemma.

\begin{lemme}
\label{free action}
The group $\GL_n(\C)$ acts freely on $\tilde{\Gie}(n,r)$.
\end{lemme}
\begin{proof}
All points of $\tilde{\Gie}(n,r)$ are stable thanks to \cite[Theorem $10.34$]{Kir06}. The \cite[Corollary $2.3.8$]{Ginz08} gives that $\GL_n(\C)$ acts freely on all stable points of $\mathcal{R}(n,r)$.
\end{proof}

\begin{prop}
The Gieseker space is a smooth algebraic variety of dimension $2nr$.
\end{prop}
\begin{proof}
Since the stability parameter is generic, we can apply \cite[Theorem $10.35$]{Kir06}, which directly gives the desired result.
\end{proof}

\noindent The Gieseker space is naturally equipped with a $(\GL_2(\C)\times \GL_r(\C))$-action. For all $(g,h,(\alpha,\beta,v^1,v^2))$ in $\GL_2(\C) \times \GL_r(\C) \times \tilde{\Gie}(n,r),$
\begin{equation*}
^{(g,h)}(\alpha,\beta,v^1,v^2):=(a\alpha + b \beta, c\alpha + d \beta, v^1 h^{-1}, \mathrm{det}(g) hv^2)
\end{equation*}
where $g=\begin{pmatrix} a & b \\ c & d \end{pmatrix}$. This action commutes with the $\GL_n(\C)$-action defined in (\ref{GLn act}) and induces an action on $\Gie(n,r)$.

\vspace*{0.25cm}

\noindent Let us now take $\ell \geq 1$ an integer and denote by $\Gamma$ the cyclic group of order $\ell$ contained in the maximal diagonal torus $\T_1$ of $\SL_2(\C)$. Let ${\omega_{\ell}:=\mathrm{diag}(\zeta_{\ell},\zeta_{\ell}^{-1}) \in \Gamma}$, where $\zeta_{\ell}$ is a primitive $\ell^{\text{th}}$ root of unity. Denote by $I_{\Gamma}$ the set of all irreducible characters of $\Gamma$. This set is finite since $\Gamma$ is a finite group. For each irreducible character $\chi$ of $\Gamma$, let $X_{\chi}$ be the $\Gamma$-module of character $\chi$. Let $\chi_0 \in I_{\Gamma}$ denote the trivial character and more generally define
\begin{equation*}
\chi_i: \begin{array}{ccc}
\Gamma & \to & \C \\
\omega_{\ell}^k & \mapsto & \zeta_{\ell}^{ik}
\end{array}, \qquad \forall i \in \Z/\ell\Z.
\end{equation*}

\noindent Let $X_{\std}$ be the $\Gamma$-module obtained as the restriction to $\Gamma$ of the standard $\GL_2(\C)$ representation $\C^2$. Let $\chi_{\std}$ denote its character. Note that $\chi_{\std}=\chi_{\smin 1}+\chi_{1}$.

\noindent Fix $\ssf \in \left(\mathbb{Z}/\ell\mathbb{Z}\right)^r$ and denote by ${t_{\ssf}:=\mathrm{diag}(\zeta_{\ell}^{\ssf_i}) \in \GL_r(\C)}$. One can then consider the following  group morphism
\begin{equation*}
\sigma^{\ssf}_r:
\begin{array}{ccc}
\Gamma & \to & \GL_r(\C)\\
\omega_{\ell} & \mapsto & t_{\ssf}.
\end{array}
\end{equation*}
This morphism equips $\C^r$ with a structure of $\Gamma$-module. Denote this $\Gamma$-module by $N^{\ssf}$. Furthermore, denote by $\Gamma_{\ssf}$ the cyclic subgroup of $\SL_2(\C) \times \GL_r(\C)$ generated by $(\omega_{\ell}, t_{\ssf})$.

\section{ Irreducible components of $\Gie(n,r)^{\Gamma_{\ssf}}$}
\noindent In this section we will study the irreducible components of the fixed points locus in the Gieseker space under $\Gamma_{\ssf}$. The goal will be to give a description of these irreducible components as Nakajima quiver varieties over the framed McKay quiver associated with $\Gamma$.
This quiver has $I_{\Gamma}$, the set of all irreducible characters of $\Gamma$, as a set of vertices. For $(\chi,\chi') \in I_{\Gamma}^2$, there is an arrow from $\chi$ to $\chi'$ whenever $\langle \chi_{\std}\chi,\chi' \rangle \neq 0$. Since $X_{\std}$ is isomorphic to its dual as a $\Gamma$-module, if there is an arrow from $\chi$ to $\chi'$, then there is an arrow from $\chi'$ to $\chi$. The framed McKay quiver is the McKay where we have added an extra framing vertex for each original vertex. For each original vertex, we have also added two arrows. One from this vertex to its associated framing vertex and another from the framing vertex to the original one.

\begin{ex}
When $\ell=3$, the framed McKay quiver is
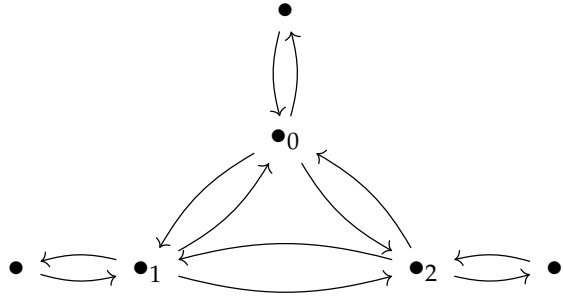
\begin{center}
\begin{tikzcd}[row sep=large]
&& \bullet \ar[bend right=15]{d} && \\
&& \bullet_0 \ar[bend left=-15]{dl} \ar[bend right=15]{dr} \ar[bend right=15]{u} && \\
\bullet \ar[bend right=15]{r} & \bullet_1 \ar[bend right=15]{ru} \ar[bend right=15]{rr} \ar[bend right=15]{l}&& \bullet_2 \ar[bend right=15]{lu} \ar[bend right=15]{ll} \ar[bend right=15]{r} & \bullet \ar[bend right=15]{l}
\end{tikzcd}
\end{center}
\noindent where the vertices are labelled by $\Z/\ell\Z \simeq I_{\Gamma}$.
\end{ex}

\noindent In the following two subsections, we will first explain how to construct a representation of the framed McKay quiver out of a fixed point in the Gieseker space. We will then go the other way around i.e. take a representation of the framed McKay quiver and construct a representation of the framed Jordan quiver.

\subsection{Deconstruction}

\noindent Let us start with $\overline{(\alpha,\beta,v^1,v^2)} \in {\Gie(n,r)}^{\Gamma_{\ssf}}$. Let $\tilde{A}_{\alpha,\beta} = e_1 \otimes \alpha + e_2 \otimes \beta$ be an element of $\C^2 \otimes \mathrm{End}(\C^n)$ and denote by $A_{\alpha,\beta}$ the image of $\tilde{A}_{\alpha,\beta}$ though this chain of canonical isomorphisms
\begin{equation*}
\C^2 \otimes \mathrm{End}(\C^n) \iso \Hom(\C,\C^2) \otimes \Hom(\C^n,\C^n) \iso \Hom(\C^n, \C^2 \otimes \C^n).
\end{equation*}

\noindent When no confusion is possible we will write $A$ instead of $A_{\alpha,\beta}$.
For each $x \in \Gamma_{\ssf}$, thanks to Lemma \ref{free action}, there exists a unique $g_{x} \in \GL(\C^n)$ such that
\begin{equation}
\label{equ 1}
^{x}(\alpha, \beta, v^1,v^2) = g_{x}.(\alpha, \beta,v^1,v^2).
\end{equation}

\noindent Consider now the following group morphism
\begin{equation*}
\sigma^{\ssf}_n:
\begin{array}{ccc}
\Gamma & \to & \GL_n(\C)\\
\omega_{\ell} & \mapsto & g_{(\omega_{\ell},t_{\ssf})}^{-1}  \\
\end{array}.
\end{equation*}
This morphism equips $\C^n$ with a structure of $\Gamma$-module. Denote this $\Gamma$-module by $M^{\ssf}$.
Consider the following isomorphism of groups:
\begin{equation*}
D_{\Gamma}^{\ssf}:\begin{array}{ccc}
\Gamma & \to & \Gamma_{\ssf}\\
\omega_{\ell} & \mapsto & (\omega_{\ell},t_{\ssf})  \\
\end{array}.
\end{equation*}

\noindent One can reformulate $(\ref{equ 1})$ as follows
\begin{equation*}
\label{fix_eq}
^{D^{\ssf}_{\Gamma}(\gamma)}(\alpha, \beta, v^1,v^2) = \sigma^{\ssf}_n(\gamma^{-1}).(\alpha, \beta,v^1,v^2), \qquad \forall \gamma \in \Gamma.
\end{equation*}

\begin{lemme}
The morphism $A:M^{\ssf} \to X_{\mathrm{std}}\otimes M^{\ssf}$ is $\Gamma$-equivariant.
\end{lemme}
\begin{proof}
For $\omega_{\ell}=\begin{pmatrix}\zeta_{\ell} & 0 \\0 &\zeta_{\ell}^{-1}\end{pmatrix} \in \Gamma$, using equation (\ref{equ 1}), one has the following equalities
\begin{center}
$\begin{cases}
\alpha g_{D^{\ssf}_{\Gamma}(\gamma)}^{-1}  = \zeta_{\ell} g_{D^{\ssf}_{\Gamma}(\gamma)}^{-1}\alpha\\
\beta g_{D^{\ssf}_{\Gamma}(\gamma)}^{-1} =  \zeta_{\ell}^{-1} g_{D^{\ssf}_{\Gamma}(\gamma)}^{-1}\beta
\end{cases}$
\end{center}
which then gives
\begin{center}
$\begin{cases}
e_1 \otimes \alpha g_{D^{\ssf}_{\Gamma}(\gamma)}^{-1}  = \zeta_{\ell}e_1 \otimes g_{D^{\ssf}_{\Gamma}(\gamma)}^{-1} \alpha\\
e_2 \otimes \beta g_{D^{\ssf}_{\Gamma}(\gamma)}^{-1} = \zeta_{\ell}^{-1}e_2 \otimes  g_{D^{\ssf}_{\Gamma}(\gamma)}^{-1}\beta
\end{cases}$
\end{center}
Summing these two equations provides exactly that $A$ is $\Gamma$-equivariant.
\end{proof}

\noindent Consider the morphism
\begin{equation*}
\mathrm{det}: \begin{array}{ccc}
\Xstd \otimes \Xstd & \to_{\Gamma} & X_{\chi_0}\\
\begin{pmatrix} r_1 \\ r_2 \end{pmatrix} \otimes \begin{pmatrix} s_1 \\ s_2 \end{pmatrix} & \mapsto & r_1s_2-r_2s_1 \\
\end{array}.
\end{equation*}

\noindent Define $\Delta^A: \Xstd \otimes M^{\ssf} \to M^{\ssf}$ as the composition of the two following maps
\begin{center}
\begin{tikzcd}
\Xstd \otimes M^{\ssf} \ar[r, "\mathrm{Id} \otimes A"] & \Xstd \otimes \Xstd \otimes M^{\ssf} \ar[r, "\mathrm{det} \otimes \mathrm{Id}"] & M^{\ssf}
\end{tikzcd}
\end{center}

\begin{lemme}
The morphism $\Delta^A$ is $\Gamma$-equivariant.
\end{lemme}
\begin{proof}
This follows from the equivariance of $A$ and of $\mathrm{det}$.
\end{proof}

\noindent In general if $M$ and $N$ are two $\Gamma$ modules, let $\Hom_{\Gamma}(M,N)$ be the set of all $\Gamma$-equivariant morphisms between $M$ and $N$.
\begin{deff}
From the point of view of $\Gamma$-modules, the representation space of the doubled, framed McKay quiver of $\Gamma$ associated with the $\Gamma$-modules $M$ and $N$ is
\begin{equation*}
\RepQ^{\Gamma}_{M,N}:=\Hom_{\Gamma}(X_{\std}\otimes M,M) \oplus \Hom_{\Gamma}(N,M) \oplus \Hom_{\Gamma}(M,N).
\end{equation*}
\noindent The $\Gamma$-module $M$ can be thought of as the "dimension vector" for the McKay quiver and $N$ as the "dimension vector" for the framing.
\end{deff}

\noindent Note that, $v^1$ defines an element $Z_1^v$ of $\Hom_{\Gamma}(N^{\ssf},M^{\ssf})$. Indeed, thanks to $(\ref{equ 1})$:
\begin{equation*}
v^1t_{\ssf}=\sigma_s(\omega_{\ell})v^1.
\end{equation*}
In the same way, $v^2$ defines an element $Z_2^v$ of $\Hom_{\Gamma}(M^{\ssf},N^{\ssf})$. Assembling everything gives the following Proposition.
\begin{prop}
\label{constr rep}
For each $\overline{(\alpha,\beta,v^1,v^2)} \in {\Gie(n,r)}^{\Gamma_{\ssf}}$, $(\Delta^{A_{\alpha,\beta}},Z_1^v,Z_2^v)$ is an element of $\RepQ^{\Gamma}_{M^{\ssf},N^{\ssf}}$.
\end{prop}

\subsection{Reconstruction}

Let us turn it the other way around. Take a $\Gamma$-module $M$ of dimension $n$.
Consider the morphism
\begin{center}
$\mathrm{ted}: \begin{array}{ccc}
X_{\chi_0}& \to_{\Gamma} & \Xstd \otimes \Xstd\\
1 & \mapsto & e_2\otimes e_1 - e_1 \otimes e_2  \\
\end{array}$\\
\end{center}
Take $(\Delta,Z_1,Z_2) \in \RepQ^{\Gamma}_{M,N^{\ssf}}$ and let $A_{\Delta}$ be the composition
\begin{center}
\begin{tikzcd}
M \ar[r, "\mathrm{ted}\otimes \mathrm{Id}"] & \Xstd \otimes \Xstd \otimes M \ar[r, "\mathrm{Id} \otimes \Delta"] & \Xstd \otimes M.
\end{tikzcd}
\end{center}
\begin{lemme}
\label{lemme_fix}
The morphism $A_{\Delta}$ is $\Gamma$-equivariant.
\end{lemme}
\begin{proof}
Since $\mathrm{ted}$ and $\Delta$ are $\Gamma$-equivariant, $A_{\Delta}$ also is $\Gamma$-equivariant.
\end{proof}

\noindent Define $(\alpha_\Delta,\beta_\Delta) \in \mathrm{End}(M)^2$ such that for all $m \in M$
\begin{equation*}
A_{\Delta}(m)=e_1 \otimes \alpha_{\Delta}(m) + e_2 \otimes \beta_{\Delta}(m)
\end{equation*}
Recall that $\RepQ^{\bullet}_{M,N^{\ssf}}:=\Hom(M,M)\oplus \Hom(M,M) \oplus \Hom(N^{\ssf},M) \oplus \Hom(M,N^{\ssf})$.
We are now able to define the following linear map
\begin{center}
$\tilde{\iota}_{M,N^{\ssf}}^{\Gamma}: \begin{array}{ccc}
\RepQ^{\Gamma}_{M,N^{\ssf}} & \to & \RepQ^{\bullet}_{M,N^{\ssf}}\\
(\Delta,Z_1,Z_2) & \mapsto & (\alpha_{\Delta} ,\beta_{\Delta},Z_1 ,Z_2)  \\
\end{array}$\\
\end{center}

\begin{lemme}
\label{iota_equiv}
The map $\tilde{\iota}_{M,N^{\ssf}}^{\Gamma}$ is $\Aut_{\Gamma}(M)$-equivariant.
\end{lemme}
\begin{proof}
Take $g \in \Aut_{\Gamma}(M)$ and $(\Delta,Z_1,Z_2) \in \RepQ^{\Gamma}_{M,N^{\ssf}}$. By definition of the action of $\Aut_{\Gamma}(M)$ on $\RepQ^{\Gamma}_{M,N^{\ssf}}$, one has that $(g.\Delta)_{e_2}=g\Delta_{e_2}g^{-1}$ and $(g.\Delta)_{e_1}=g\Delta_{e_1}g^{-1}$. Since $\alpha_{\Delta}=-\Delta_{e_2}$ and $\beta_{\Delta}=\Delta_{e_1}$, it is clear that, by definition of the $\GL(M)$-action on $\RepQ^{\bullet}_{M,N^{\ssf}}$, the map $\tilde{\iota}_{M,N^{\ssf}}^{\Gamma}$ is $\Aut_{\Gamma}(M)$-equivariant.
\end{proof}

\begin{prop}
\label{reconstr works}
The map $\tilde{\iota}_{M,N^{\ssf}}^{\Gamma}$ induces a map $\iota_{M,N^{\ssf}}^{\Gamma}: \QVG_{\boldsymbol{1}}(M,N^{\ssf}) \to \Gie(n,r)^{\Gamma_{\ssf}}$.
\end{prop}
\begin{proof}
First let us explain why $\tilde{\iota}_{M,N^{\ssf}}^{\Gamma}(\tQVG_{\boldsymbol{1}}(M,N^{\ssf})) \subset \tilde{\Gie}(n,r)$.\\
Take $(\Delta,Z_1,Z_2) \in \RepQ^{\Gamma}_{M,N^{\ssf}}$. By construction, $\alpha_{\Delta}$ and $\beta_{\Delta}$ are respectively the morphisms $-\Delta_{e_2}$ and $\Delta_{e_1}$.
One then has
\begin{equation*}
\alpha_{\Delta}\beta_{\Delta} -\beta_{\Delta}\alpha_{\Delta} + Z_1Z_2 = -\Delta_{e_2}\Delta_{e_1}+\Delta_{e_1}\Delta_{e_2} + Z_1Z_2
\end{equation*}
This computation shows that if $\mu_{\Gamma}(\Delta,Z_1,Z_2)=0$, then $ \mu_{n,r}(\alpha_{\Delta},\beta_{\Delta},Z_1,Z_2)=0$. Let $(\Delta, Z_1,Z_2)$ be an element of $\tQVG_{\boldsymbol{1}}(M,N^{\ssf})$ and $M'$ be the subspace of $M$ that is generated by $\mathrm{Im}(Z_1)$ under the action of ${\alpha_{\Delta}}$ and ${\beta_{\Delta}}$. Thanks to \cite[Example $10.36$]{Kir06}, it is enough to show that $M'=M$ to conclude that $\tilde{\iota}_{M,N^{\ssf}}^{\Gamma}(\Delta,Z_1,Z_2)$ is $\boldsymbol{1}$-semistable. The $\Gamma$-equivariance of $\Delta:X_{\mathrm{std}} \otimes M \to M$ implies that $M'$ is a $\Gamma$-submodule of $M$. By hypothesis, $(\Delta, Z_1, Z_2)$ is $\boldsymbol{1}$-semistable and $\Delta(X_{\mathrm{std}}\otimes M') \subset M'$, which gives that $M'=M$ thanks to \cite[Lemma $2.30$]{Pae1}.
  The morphism  $\tilde{\iota}_M^{\Gamma}$ is $\Aut_{\Gamma}(M)$-equivariant (Lemma \ref{iota_equiv}) . It then induces ${\iota_{M,N^{\ssf}}^{\Gamma}: \QVG_{\boldsymbol{1}}(M,N^{\ssf}) \to \Gie(n,r)}$. There remains to show that $\iota_{M,N^{\ssf}}^{\Gamma}\left(\QVG_{\boldsymbol{1}}(M,N^{\ssf})\right) \subset \Gie(n,r)^{\Gamma_{\ssf}}$. If $(\Delta,Z_1,Z_2) \in \RepQ^{\Gamma}_{M,N^{\ssf}}$ and $m \in M$, then Lemma \ref{lemme_fix} implies that
\begin{equation*}
\omega_{\ell}.(e_1\otimes \alpha_{\Delta}(m)+e_2\otimes \beta_{\Delta}(m))=e_1\otimes \alpha_{\Delta}(\omega_{\ell}.m) + e_2 \otimes \beta_{\Delta}(\omega_{\ell}.m)
\end{equation*}
From there, one has
\begin{center}
$\begin{cases}
\omega_{\ell}.(\zeta_{\ell}\alpha_{\Delta}(m))=\alpha_{\Delta}(\omega_{\ell}.m)\\
\omega_{\ell}.(\zeta_{\ell}^{-1}\beta_{\Delta}(m))=\beta_{\Delta}(\omega_{\ell}.m)
\end{cases}
\Rightarrow
\begin{cases}
\zeta_{\ell}\alpha_{\Delta}(m)=\omega_{\ell}^{-1}.\alpha_{\Delta}(\omega_{\ell}.m)\\
\zeta_{\ell}^{-1}\beta_{\Delta}(m)=\omega_{\ell}^{-1}.\beta_{\Delta}(\omega_{\ell}.m)
\end{cases}
$
\end{center}
Moreover, the $\Gamma$-equivariance of $Z_1$ and $Z_2$ imply that for all $(m,x) \in M\times N^{\ssf}$
\begin{align*}
Z_1(t_{\ssf}(x)))&=\omega_{\ell}.Z_1(x)\\
Z_2(\omega_{\ell}.m)&=t_{\ssf}(Z_2(m)).
\end{align*}
If we denote by $g \in \GL(M)$ the action of $\omega_{\ell}^{-1}$ on $M$, then
\begin{equation*}
^{(\omega_{\ell}, t_{\ssf})}(\alpha_{\Delta},\beta_{\Delta},Z_1,Z_2)=g.(\alpha_{\Delta}, \beta_{\Delta},Z_1,Z_2)
\end{equation*}
\end{proof}

\subsection{Synthesis}

Let us now connect the last two subsections. Denote by $\RepGn$ the set of all group morphisms $\sigma:\Gamma \to \GL_n(\C)$ and let $L_{\Gamma}^{\ssf}$ be the following algebraic variety
\begin{equation*}
\left\{(\alpha,\beta,v^1,v^2,\sigma) \in \tilde{\Gie}(n,r) \times \RepGn \text{ } |^{D^{\ssf}_{\Gamma}(\gamma)}(\alpha,\beta,v^1,v^2) = \sigma(\gamma^{-1}).(\alpha,\beta,v^1,v^2), \forall \gamma \in \Gamma\right \}.
\end{equation*}

\noindent Define
\begin{itemize}
\item$\tilde{p}_{\ssf}: L_{\Gamma}^{\ssf} \to \tilde{\Gie}(n,r)$,
\item $p_{\ssf}: L_{\Gamma}^{\ssf} \to \Gie(n,r) = \pi \circ \tilde{p}_{\ssf}$,
\item $q_{\ssf}: L_{\Gamma}^{\ssf} \to \RepGn$.
\end{itemize}
Let the group $\GL_n(\C)$ act by conjugacy on $\RepGn$ and diagonally on $L_{\Gamma}^{\ssf}$.
The morphisms $\tilde{p}_{\ssf}$ and $q_{\ssf}$ are $\GL_n(\C)$-equivariant.
Moreover, note that $p_{\ssf}(L_{\Gamma}^{\ssf})=\Gie(n,r)^{\Gamma_{\ssf}}$.

\begin{lemme}
\label{inj}
If $(\alpha,\beta,v^1,v^2,\sigma)$ and $(\alpha,\beta,v^1,v^2, \sigma')$ are both in $L_{\Gamma}^{\ssf}$, then $\sigma=\sigma'$
\end{lemme}
\begin{proof}
From the assumption, one has $\sigma(\gamma).(\alpha,\beta,v^1,v^2)=\sigma'(\gamma).(\alpha,\beta,v^1,v^2)$ for all $\gamma \in \Gamma$. The freedom of the $\GL_n(\C)$-action (Lemma \ref{free action}) implies that $\sigma=\sigma'$.
\end{proof}

\noindent If $\sigma \in \RepGn$, then $M^{\sigma}$ will denote the $\Gamma$-module induced by $\sigma$.

\begin{deff}
Let $\sigma \in \RepGn$ be such that $q_{\ssf}^{-1}(\sigma) \neq \emptyset$ and define
\begin{center}
$\tilde{\kappa}_{\ssf,\sigma}^{\Gamma}: \begin{array}{ccc}
q_{\ssf}^{-1}(\sigma) & \to & \RepQ^{\Gamma}_{M^{\sigma},N^{\ssf}}\\
(\alpha,\beta,v^1,v^2,\sigma) & \mapsto & (\Delta^{A_{\alpha,\beta}},Z_1^v,Z_2^v)  \\
\end{array}$
\end{center}
where $(\Delta^{A_{\alpha,\beta}},Z_1^v,Z_2^v)$ is as in Proposition \ref{constr rep}.
\end{deff}

\begin{prop}
\label{rep works}
If $\sigma \in \RepGn$ is such that $q_{\ssf}^{-1}(\sigma) \neq \emptyset$, then $\tilde{\kappa}_{\ssf,\sigma}^{\Gamma}(q_{\ssf}^{-1}(\sigma)) = \tQV_{\boldsymbol{1}}^{\Gamma}(M^{\sigma},N^{\ssf})$.
\end{prop}
\begin{proof}
Take $(\alpha,\beta,v^1,v^2,\sigma) \in q_{\ssf}^{-1}(\sigma)$, by construction $\Delta^A_{e_1}=\beta$ and $\Delta^A_{e_2}=-\alpha$. One then has
\begin{equation*}
\Delta^A_{e_1}\Delta^A_{e_2}-\Delta^A_{e_2}\Delta^A_{e_1} + Z_1^vZ_2^v = -\beta\alpha + \alpha\beta + v^1v^2
\end{equation*}
This proves that if $(\alpha,\beta,v^1,v^2,\sigma)\in q^{-1}_{\ssf}(\sigma)$ then  $\mu_{\Gamma}(\Delta^{A_{\alpha,\beta}},Z_1^v,Z_2^v)=0$. Take $M'$ a $\Gamma$-submodule of $M^{\sigma}$ such that $\Delta^{A}(\Xstd \otimes M') \subset M'$ and $\mathrm{Im}(Z_1^v) \subset M'$. Thanks to \cite[Lemma $2.30$]{Pae1}, it is enough to show that $M^{\sigma} \subset M'$ to conclude that $(\Delta^A,Z_1^v,Z_2^v)$ is {$\boldsymbol{1}$-semistable}. Since the element $(\alpha,\beta,v^1,v^2)$ is in  $\tilde{\Gie}(n,r)$, the vector space $M^{\sigma}$ is generated by $\mathrm{Im}(v^1)$ under the action of $\alpha$ and $\beta$. Since $\alpha = -\Delta^A_{e_2}$, $\beta= \Delta^A_{e_1}$ and $\mathrm{Im}(v^1) \subset M'$, one can conclude that $M^{\sigma}\subset M'$. Thus $\tilde{\kappa}_{\ssf,\sigma}^{\Gamma}(q^{-1}(\sigma)) \subset \tQV_{\boldsymbol{1}}^{\Gamma}(M^{\sigma},N^{\ssf})$.\\
To finish, let us prove the other inclusion. Pick $(\Delta,Z_1,Z_2) \in \tQVG_{\boldsymbol{1}}(M^{\sigma},N^{\ssf})$. A quick computation gives
\begin{align*}
^{D^{\ssf}_{\Gamma}(\omega_{\ell})}(-\Delta_{e_2},\Delta_{e_1},Z_1,Z_2)
&= (-\Delta_{\omega_{\ell}^{-1}e_2}, \Delta_{\omega_{\ell}^{-1}e_1},Z_1t_{\ssf}^{-1},t_{\ssf}Z_2)\\
&= (-\sigma(\omega_{\ell})^{-1}\Delta_{e_2}\sigma(\omega_{\ell}), \sigma(\omega_{\ell})^{-1}\Delta_{e_1}\sigma(\omega_{\ell}),\sigma(\omega_{\ell})^{-1}Z_1,Z_2\sigma(\omega_{\ell}))\\
&= \sigma(\omega_{\ell})^{-1}.(-\Delta_{e_2}, \Delta_{e_1}, Z_1,Z_2)
\end{align*}
The last equality comes from the $\Gamma$-equivariance of $\Delta$, $Z_1$ and $Z_2$.
\end{proof}

\noindent Consider now $\kappa_{\ssf,\sigma}^{\Gamma}: q_{\ssf}^{-1}(\sigma) \to \QVG_{\boldsymbol{1}}(M^{\sigma},N^{\ssf})$ which is onto whenever $q_{\ssf}^{-1}(\sigma) \neq \emptyset$, thanks to Proposition \ref{rep works}.
\begin{prop}
 \label{prop_commutes}
For each $\sigma \in \mathrm{Rep}_{\Gamma,n}$ such that $q_{\ssf}^{-1}(\sigma) \neq \emptyset$, the following diagram commutes
\begin{center}
\begin{tikzcd}[column sep=huge]
 q_{\ssf}^{-1}(\sigma) \ar[d, "p"] \ar[r, "\kappa_{\ssf,\sigma}^{\Gamma}"] & \QVG_{\boldsymbol{1}}(M^{\sigma},N^{\ssf}) \ar[dl, "\iota_{M^{\sigma},N^{\ssf}}^{\Gamma}"]\\
 \Gie(n,r)^{\Gamma_{\ssf}}
 \end{tikzcd}
 \end{center}
 \end{prop}
 \begin{proof}
Take $(\alpha, \beta,v^1,v^2, \sigma) \in q_{\ssf}^{-1}(\sigma)$, by construction one has
\begin{equation*}
\tilde{\iota}_{M^{\sigma},N^{\ssf}}^{\Gamma}(\tilde{\kappa}_{\ssf,\sigma}^{\Gamma}(\alpha,\beta,v^1,v^2,\sigma))=(\alpha,\beta,v^1,v^2)
\end{equation*}
 which shows that the diagram commutes.
\end{proof}
\noindent Denote by $\Delta_{\Gamma}:=\Z I_{\Gamma}$. For $\sigma \in \mathrm{Rep}_{\Gamma,n}$ and for $\chi \in I_{\Gamma}$, let $d^{\sigma}_{\chi} := \mathrm{dim}(\Hom_{\Gamma}(X_{\chi},M^{\sigma}))$.
Denote the character map by
\begin{equation*}
\mathrm{char}_{\Gamma}: \begin{array}{ccc}
\RepGn & \to & \Delta_{\Gamma}\\
\sigma & \mapsto & \sum_{\chi \in I_{\Gamma}}{d^{\sigma}_{\chi}\chi}  \\
\end{array}
\end{equation*}

\noindent Let $\tilde{\mathcal{A}}^n_{\Gamma} := \mathrm{char}_{\Gamma}(\RepGn)$ which is just a combinatorial way to encode $\mathrm{Char}_{\Gamma,n}$ the set of all characters associated with $n$-dimensional representations of $\Gamma$. Moreover for ${d \in \tilde{\mathcal{A}}^n_{\Gamma}}$, denote by $\mathcal{C}_d:=\mathrm{char}_{\Gamma}^{-1}(d)$ the set of all $n$-dimensional representations that have character $\sum_{\chi \in I_{\Gamma}}{d_{\chi}\chi}$. Note that if one takes $\sigma \in \mathcal{C}_d$, then $\mathcal{C}_d$ is just the $\mathrm{GL}_n(\C)$-conjugacy class of $\sigma$. With that notation, one has
\begin{equation*}
\RepGn = \coprod_{d \in \tilde{\mathcal{A}}^n_{\Gamma}}{\mathcal{C}_d}
\end{equation*}

\noindent Denote by $\mathcal{A}^n_{\Gamma,\ssf}:=\{d \in \tilde{\mathcal{A}}^n_{\Gamma} | q_{\ssf}^{-1}(\mathcal{C}_d) \neq \emptyset\}$. For $d \in \mathcal{A}^n_{\Gamma,\ssf}$, one can associate an $I_{\Gamma}$-graded vector space $V^d:=\bigoplus_{\chi \in I_{\Gamma}}{\C^{d_{\chi}}}$. Denote by ${M^d:=\bigoplus_{\chi \in I_{\Gamma}}{V^d_{\chi}\otimes X_{\chi}}}$ the $\Gamma$-module associated with $d$.


\begin{deff}
Take $d \in \mathcal{A}^n_{\Gamma,\ssf}$. Let us define the variety $\mathcal{M}_d^{\Gamma_{\ssf}}:=p_{\ssf}(q_{\ssf}^{-1}(\mathcal{C}_d))$.
\end{deff}

\begin{rmq}
Note that $\mathcal{M}_d^{\Gamma_{\ssf}}=p_{\ssf}(q_{\ssf}^{-1}(\sigma))$ for any $\sigma \in \mathcal{C}_d$ because $\tilde{p}_{\ssf}$ and $q_{\ssf}$ are $\GL_n(\C)$-equivariant morphisms.
\end{rmq}

\begin{thm}
\label{decomp}
For each non-zero integers $n$, $r$ and $\ell$ and for each $\ssf \in \left(\Z/\ell\Z\right)^r$, the scheme of $\Gamma_{\ssf}$-fixed points in the Gieseker space decomposes into irreducible components
\begin{equation*}
\Gie(n,r)^{\Gamma_{\ssf}} = \coprod_{d \in \mathcal{A}^n_{\Gamma, \ssf}}{\QV_d^{\Gamma_{\ssf}}}
\end{equation*}
\end{thm}
\begin{proof}
Take $d \in \mathcal{A}^n_{\Gamma, \ssf}$. Let us first show that $\QV_d^{\Gamma_{\ssf}}$ is an irreducible and closed set of $\Gie(n,r)^{\Gamma_{\ssf}}$.
Take $\sigma \in \mathcal{C}_d$. The \cite[Proposition $2.28$]{Pae1} implies that $\QVG_{\boldsymbol{1}}(M^{\sigma},N^{\ssf})$ is irreducible and since $\QV_d^{\Gamma_{\ssf}}=\iota_{M^{\sigma},N^{\ssf}}^{\Gamma}(\QVG_{\boldsymbol{1}}(M^{\sigma},N^{\ssf}))$, the scheme $\QV_d^{\Gamma}$ is irreducible.
To show that $\QV^{\Gamma}_d$ is a closed subset of $\Gie(n,r)^{\Gamma_{\ssf}}$, note that $L_{\Gamma}^{\ssf}$ is connected and that $\tilde{p}_{\ssf}$ is injective thanks to Lemma \ref{inj}. The image of $p_{\ssf}$ being equal to $\Gie(n,r)^{\Gamma_{\ssf}}$, this implies that
\begin{equation*}
\tilde{p}_{\ssf}(L_{\Gamma}^{\ssf})= \pi^{-1}(\Gie(n,r)^{\Gamma_{\ssf}})
\end{equation*}
The group $\GL_n(\C)$ acting freely on $\tilde{\Gie}(n,r)$ (Lemma \ref{free action}), the morphism $\pi$ is smooth. The group $\Gamma$ being a finite group, one knows that $\Gie(n,r)^{\Gamma_{\ssf}}$ is smooth and in particular $\pi^{-1}(\Gie(n,r)^{\Gamma_{\ssf}})$ is smooth. Thus, the morphism $\tilde{p}_{\ssf}:L_{\Gamma}^{\ssf} \to \pi^{-1}(\Gie(n,r)^{\Gamma_{\ssf}})$ becomes an isomorphism of algebraic varieties.
Denote by ${p'_{\ssf}:\pi^{-1}(\Gie(n,r)^{\Gamma_{\ssf}}) \to L_{\Gamma}^{\ssf}}$ its inverse. Consider now
\begin{center}
$\tilde{\tau}_{\ssf}:=q_{\ssf} \circ p'_{\ssf} : \pi^{-1}(\Gie(n,r)^{\Gamma_{\ssf}}) \to \RepGn$.
\end{center}
The morphism $\tilde{\tau}_{\ssf}$ is $\GL_n(\C)$-equivariant. Therefore, it induces a morphism
\begin{equation*}
{\tau_{\ssf}: \Gie(n,r)^{\Gamma_{\ssf}} \to \RepGn\sslash\GL_n(\C)}.
\end{equation*}

\noindent The big picture draws like so
\vspace{0.5cm}
\begin{center}
$\begin{tikzcd}[row sep=large, column sep=large]
L_{\Gamma}^{\ssf} \ar[r, "q_{\ssf}"] \ar[d, shift right, "\tilde{p}_{\ssf}"'] \ar[dr, "p_{\ssf}"'] & \RepGn \ar[r, two heads] & \RepGn\sslash \mathrm{GL}_n(\C)\\
\pi^{-1}(\Gie(n,r)^{\Gamma_{\ssf}}) \ar[u,shift right, "p'_{\ssf}"'] \ar[r,"\pi"']& \Gie(n,r)^{\Gamma_{\ssf}} \ar[ur, "\tau_{\ssf}"']
\end{tikzcd}$
\end{center}
\vspace{0.5cm}

\noindent It is then clear that
\begin{equation*}
\tau_{\ssf}^{-1}(\mathcal{C}_d) = p_{\ssf}(q_{\ssf}^{-1}(\mathcal{C}_d))=  \QV_d^{\Gamma_{\ssf}}
\end{equation*}
which proves that $\QV_d^{\Gamma}$ is a closed subset. Indeed $\mathcal{C}_d$ is closed because all representations of $\Gamma$ are semisimple since $\Gamma$ is a finite group.\\
Finally, one has to show that $\Gie(n,r)^{\Gamma_{\ssf}}=\bigcup_{d \in \mathcal{A}^n_{\Gamma,\ssf}}{\QV_d^{\Gamma_{\ssf}}}$ which comes for free
 \begin{equation*}
 \Gie(n,r)^{\Gamma_{\ssf}} = \bigcup_{d \in \mathcal{A}^n_{\Gamma,\ssf}}{\tau_{\ssf}^{-1}(\mathcal{C}_d}) = \bigcup_{d \in \mathcal{A}^n_{\Gamma,\ssf}}{\QV_d^{\Gamma_{\ssf}}}
 \end{equation*}
\end{proof}

\begin{prop}
\label{prop_iso_quiver_var}
For each $d \in \mathcal{A}^n_{\Gamma, \ssf}$ and for each $\sigma \in \mathcal{C}_d$
\begin{equation*}
\iota_{M^{\sigma},N^{\ssf}}^{\Gamma}: \QVG_{\boldsymbol{1}}(M^{\sigma},N^{\ssf}) \to \QV_d^{\Gamma_{\ssf}}
\end{equation*}
is an isomorphism of algebraic varieties.
\end{prop}
\begin{proof}
The morphism $\iota^{\Gamma}_{M^{\sigma},N^{\ssf}}$ is injective. Using Proposition \ref{rep works} and \ref{prop_commutes}, we have
\begin{equation*}
\QV_d^{\Gamma_{\ssf}}=\iota_{M^{\sigma},N^{\ssf}}^{\Gamma}\left(\QVG_{\boldsymbol{1}}(M^{\sigma},N^{\ssf})\right)
\end{equation*}
Moreover, the variety $\QV_d^{\Gamma_{\ssf}}$ is smooth. Indeed, thanks to Theorem \ref{decomp}, one has that $\QV_d^{\Gamma_{\ssf}}$ is an irreducible component of the smooth scheme $\Gie(n,r)^{\Gamma_{\ssf}}$. Furthermore, since $\Gie(n,r)^{\Gamma_{\ssf}}$ has a finite number of irreducible components, this implies that
 $\QV_d^{\Gamma_{\ssf}}$ is open.
Finally, \cite[Theorem $2.27$ \& Proposition $2.28$]{Pae1} imply that $\QVG_{\boldsymbol{1}}(M^{\sigma},N^{\ssf})$ is connected. Summing it all up, one can now conclude that $\iota_{\boldsymbol{1},M^{\sigma},N^{\ssf}}^{\Gamma}$ is an isomorphism of algebraic varieties.
\end{proof}

\section{Combinatorics}
In this section we will give a combinatorial description of $\mathcal{A}^n_{\Gamma,\hspace*{0.05cm}\mathcalb{s}}$. It will give a simple expression of the dimension of the irreducible components of $\Gie(n,r)^{\Gamma_{\ssf}}$. Then we will explain how we recover the combinatorics of $\ell$-cores of $\ell$-charged multipartitions introduced by Fayers \cite{Fay19} and developed by Jacon and Lecouvey \cite{JL19}.\\
Before going into the combinatorics, let us start this section by introducing a bit of notation. Let $C_{\Gamma}$ be the generalized Cartan matrix of affine Lie type $A_{\ell}$. Take  $\left((\alpha_i,\alpha_i^{\vee})_{i \in \Z/\ell\Z},\mathfrak{h}_{\Gamma}\right)$ a realization (in the sense of \cite[Chapter $1$]{kac90}) of $C_{\Gamma}$ where $\mathfrak{h}_{\Gamma}$ is a $\C$-vector space of dimension $\ell+1$. For each $i \in \Z/\ell\Z$, $\alpha_i^{\vee} \in \mathfrak{h}_{\Gamma}$ and $\alpha_i \in \mathfrak{h}_{\Gamma}^*$. Such a realization can be explicitly constructed out of the McKay graph of $\Gamma$ (see \cite[Remark 2.12]{Pae1}). Let $\mathfrak{g}_{\Gamma}$ be the affine Lie algebra constructed out of this realization. Denote by $\langle,\rangle$ the symmetric bilinear form on $\mathfrak{h}_{\Gamma}^*$ \cite[§2.1]{kac90} induced by $C_{\Gamma}$. Note that since $C_{\Gamma}$ is symmetric, this is just given by the coefficients of $C_{\Gamma}$. Denote by $Q_{\Gamma}:=\Z\{\alpha_i| i \in \Z/\ell\Z\}$ and by $Q^+_{\Gamma}:=\Z_{\geq 0}\{\alpha_i| i \in \Z/\ell\Z\}$. We will identify $Q_{\Gamma}$ with $\Delta_{\Gamma}$ and $Q^+_{\Gamma}$ with $\Delta^+_{\Gamma}$. Let $P_{\Gamma}$ denote the lattice generated by the fundamental weights $(\Lambda_i)_{i \in \Z/\ell\Z}$ attached to the realization.
Let $\delta^{\Gamma}$ be the null root. For $\alpha \in Q_{\Gamma}$, denote $\alpha_i:=\langle \alpha, \Lambda_i \rangle \in \Z$.
\begin{deff}
If $\alpha \in Q_{\Gamma}$, then the size of $\alpha$ is defined as
\index{$\vert\alpha\vert_{\tilde{T}}$} $|\alpha|_{\Gamma}:=\sum_{i \in \Z/\ell\Z}{\alpha_i}$.
\end{deff}
\begin{deff}
Denote by  $w^{\ssf}_i=\#\{j \in \llbracket 1, r\rrbracket | \ssf_j= i \in \Z/\ell\Z\}$, for all $i \in \Z/\ell\Z$ and by $w^{\ssf}=\sum_{i\in \Z/\ell\Z}{w^{\ssf}_i\alpha_i}$.
\end{deff}

\noindent For $x,y \in \mathfrak{h}^*$, $x \equiv y [\delta^{\Gamma}]$ means that there exists $k \in \mathbb{Z}$ such that $x=y+k\delta^{\Gamma}$.

\subsection{Dimension of the irreducible components}

\noindent For $(d,d') \in Q_{\Gamma}^2$, let $d.d':=\sum_{i \in \Z/\ell\Z}{d_id'_i}\in \Z$. Thanks to \cite[Theorem $10.35$]{Kir06}, the irreducible component $\QV_d^{\Gamma_{\ssf}}$ has dimension $2d.w^{\ssf}-\langle d,d \rangle$ for each $d \in \mathcal{A}^n_{\Gamma,\ssf}$.

\noindent For each dominant weight $\Lambda$ of $\mathfrak{g}_{\Gamma}$, there exists an integrable highest weight module $L(\Lambda)$ of $\mathfrak{g}_{\Gamma}$ , the Kac-Moody Lie algebra associated with $\Gamma$ \cite[Lemma $10.1$]{kac90}. Denote by $P(\Lambda)$ the weights of this module.

\begin{deff}
Let $\Lambda^{\ssf}:=\sum_{i\in \Z/\ell\Z}{w^{\ssf}_i\Lambda_i}$ be the dominant weight associated with $\ssf$.
\end{deff}

\begin{lemme}
\label{comb_descr}
We have the following equality:
\begin{equation*}
\mathcal{A}_{\Gamma,\ssf}^n=\left\{d \in Q^+_{\Gamma} \big| |d|_{\Gamma}=n, \Lambda^{\ssf} - d \in P(\Lambda^{\ssf})\right\}.
\end{equation*}
\end{lemme}
\begin{proof}
Thanks to \cite[Theorem $10.2$]{Nak98} (or \cite[Theorem $13.19$]{Kir06}), the variety $\QVG_{\boldsymbol{1}}(M^{\sigma},N^{\ssf})$ is nonempty if and only if $\Lambda^{\ssf}-d \in P(\Lambda^{\ssf})$, where $d$ is the character of the $\Gamma$-module $M^{\sigma}$. By definition, $\mathcal{A}_{\Gamma,\ssf}^n$ is the set of all elements of $Q^+_{\Gamma}$ such that $|d|_{\Gamma}=n$ and such that the variety $M_d^{\Gamma_{\ssf}}$ is nonempty. Proposition \ref{prop_iso_quiver_var} concludes the proof.
\end{proof}

\noindent Denote by $P^{++}(\Lambda^{\ssf})$ the finite set of all maximal, dominant weights of $L(\Lambda^{\ssf})$. Let $W_{\Gamma}$ be the Weyl group of $\mathfrak{g}_{\Gamma}$. It acts by reflection on $\mathfrak{h}_{\Gamma}^*$. Denote this action by $*$. Using \cite[Corollary $10.1$ and $(12.6.1)$]{kac90}, we know that for each $d \in P(\Lambda^{\ssf})$, there exists a unique couple $(\Lambda^+_d,k_d) \in P^{++}(\Lambda^{\ssf})\times \mathbb{Z}_{\geq 0}$ and there exists ${\omega \in W_{\Gamma}}$ such that
\begin{equation*}
\omega*d=\Lambda^+ - k \delta^{\Gamma}.
\end{equation*}

\begin{deff}
\label{deff . action}
Let us define a new action of $W_{\Gamma}$ on $\mathfrak{h}^*_{\Gamma}$ as follows:
\begin{equation*}
\omega.d=\Lambda^{\ssf}-\omega*(\Lambda^{\ssf}-d), \qquad (\omega,d) \in W_{\Gamma} \times \mathfrak{h}^*_{\Gamma}.
\end{equation*}
\end{deff}

\noindent Thanks to Lemma \ref{comb_descr}, we have now that for each $d \in \mathcal{A}_{\Gamma,\ssf}^n$, there exists a unique pair $\left(\Lambda^+_d,k_d\right)\in P^{++}(\Lambda^{\ssf}) \times \Z_{\geq 0}$ and there exists $\omega \in W_{\Gamma}$ such that:
\begin{equation*}
\omega.d=\alpha_d + k_d \delta^{\Gamma}
\end{equation*}
where $\alpha_d=\Lambda^{\ssf}-\Lambda^+_d \in Q_{\Gamma}$ thanks to \cite[Lemma 11.2]{kac90}.

\begin{prop}
\label{dim easy}
For each $d \in \mathcal{A}_{\Gamma,s}^n$, we have the following equality:
\begin{equation*}
\mathrm{dim}\left(\mathcal{M}_d^{\Gamma_{\ssf}}\right)= 2 k_d r +  2\alpha_d.w^{\ssf} - \langle \alpha_d, \alpha_d\rangle.
\end{equation*}
\end{prop}
\begin{proof}
Take $\omega \in W_{\Gamma}$ such that $\omega.d= \alpha_d + k_d \delta^{\Gamma}$. Using the isomorphism of Lusztig-Maffei-Nakajima, we have that $\mathcal{M}_{\boldsymbol{1}}^{\Gamma}(d,w^{\ssf})$ is isomorphic to ${\mathcal{M}_{\omega*\boldsymbol{1}}^{\Gamma}(\alpha_d + k_d \delta^{\Gamma},w^{\ssf})}$. This gives that
\begin{center}
\begin{align*}
\mathrm{dim}\left(\mathcal{M}_d^{\Gamma_{\ssf}}\right)&= 2(\alpha_d + k_d \delta^{\Gamma}).w^{\ssf} - \langle \alpha_d + k_d \delta^{\Gamma}, \alpha_d + k_d \delta^{\Gamma}\rangle\\
  &= 2 \alpha_d.w^{\ssf} + 2 k_d r - \langle \alpha_d, \alpha_d\rangle - 2 \langle \alpha_d, k_d \delta^{\Gamma} \rangle\\
\end{align*}
\end{center}
Since $\alpha_d \in Q_{\Gamma}, \langle \alpha_d, \delta^{\Gamma} \rangle=0$.
\end{proof}

\begin{rmq}
Note that $2\alpha_d.w^{\ssf} - \langle \alpha_d, \alpha_d\rangle= \mathrm{dim}\left(\mathcal{M}_{\boldsymbol{1}}^{\Gamma}(\alpha_d,w^{\ssf})\right) \geq 0$.
\end{rmq}

\subsection{Multipartitions}

\begin{deff}
A partition of $n$ is a non-increasing function $\lb:\Z_{\geq 0} \to \Z_{\geq 0}$ such that $\sum_{i=0}^{\infty}{\lb_i}=n$.
\end{deff}

\noindent There is a unique partition of $0$ that we will denote by $\emptyset$. The length of $\lb$ is equal to the smallest $k \in \Z_{\geq 0}$ such that $\lb_{k}=0$ and will be denoted by $l(\lb)$. For the sake of clarity, we will often denote a partition $\lb$ by the sequence of the first $l(\lb)$ nonzero values of $\lb$.


\begin{deff}
A $r$-multipartition $\lambda^{\bullet}:=(\lambda^1,\ldots,\lambda^r)$ is the data of $r$ partitions denoted by $\lambda^j$ for $j \in \llbracket 1, r\rrbracket$. The size of a $r$-multipartition is $\sum_{i=1}^r{|\lambda^i|}$ and will be denoted by $|\lambda^{\bullet}|$.
\end{deff}

\noindent Denote by $\mathcal{P}^r_n$ the set of all $r$-multipartitions of size $n$ and by $\mathcal{P}^r:=\coprod_{n \geq 0}{\mathcal{P}^r_n}$. Denote moreover by ${\CP^r_n(\ell):= \mathcal{P}^r_n \times \Z/\ell\Z}$ the set of $\ell$-charged, $r$-multipartitions of size $n$ and by $\CP^r(\ell):=\mathcal{P}^r \times  \Z/\ell\Z$. From now on, when $r=1$, we will drop $r$ in the notation. For example, the set of all partitions of $n$ will be denoted by $\mathcal{P}_n$.

\noindent Take $(i,k) \in ({\Z/\ell\Z})^2$, and define
\begin{equation*}
R_i(\lb,k):=\#\left\{(a,b) \in \mathcal{Y}(\lb) \text{ }\big|\text{ } \overline{b\smin a}=i-k\in \Z/\ell\Z\right\}.
\end{equation*}
\begin{deff}
The $\ell$-residue is the following map
\begin{center}
$\Res_{\ell}:\begin{array}{ccc}
\CP_n^r(\ell) & \to & Q_{\Gamma}\\
(\lbb,\ssf)& \mapsto & \sum_{i \in \Z/\ell\Z}{\left(\sum_{j=1}^r{R_i(\lb^j,\ssf_j)}\right)\alpha_i}
\end{array}.$
\end{center}
\end{deff}

\noindent Young diagrams can also be defined for $r$-multipartitions. For $\lbb \in \mathcal{P}^r_n$, let
\begin{equation*}
{\mathcal{Y}(\lbb):=\left\{(a,b,c) \in \mathbb{Z}_{\geq 0}^2\times \llbracket 1, r \rrbracket| a < \lambda^c_b, b \le l(\lb^c) \right \}}.
\end{equation*}
An element of $\mathcal{Y}(\lbb)$ will be called a node of $\lbb$. A node $(a,b,c)$ of $\lbb$ is removable if $\lb^c$ stays a partition after removing the node $(a,b)$. Likewise, a node $(a,b,c)$ is addable if $\lb^c$ stays a partition after adding the node $(a,b)$. For $i \in \Z/\ell\Z$, an $i$-node of ${(\lbb,\ssf) \in \CP^r_n(\ell)}$ is a node $(a,b,c)$ of $\lbb$ such that ${\overline{b-a}=i-\ssf_c \in \Z/\ell\Z}$.

\begin{deff}
We define a $W_{\Gamma}$-action on $\CP^r_n(\ell)$ by letting the generators $(s_i)_{i\in \Z/\ell\Z}$ act as follows:
$s_i.(\lbb,\ssf)=(s_i.\lbb,\ssf)$, where $s_i.\lbb$ is the multipartition obtained from $\lbb$ where all removable $i$-nodes have been removed and all addable $i$-nodes have been added.
\end{deff}

\begin{rmq}
This action does not depend on whether we start by adding or removing nodes.
\end{rmq}

\begin{prop}
\label{equi res}
If $W_{\Gamma}$ acts on $Q_{\Gamma}$ though the $.$ action (cf. Definition \ref{deff . action}), then the map $\Res_{\ell}:\CP^r_n(\ell) \to Q_{\Gamma}$ is $W_{\Gamma}$-equivariant.
\end{prop}
\begin{proof}
By construction of the action, it is enough to prove it for $r=1$. This follows directly using abaci.
\end{proof}

\subsection{Abaci}

Abaci are a very handy way of working with charged $r$-multipartitions and a lot of things can be deduced and read from them. Let us first introduce the notion.

\begin{deff}
A $r$-abacus is a map $A: \Z \times \llbracket 1,r\rrbracket \to \{0,1\}$ verifying the following condition:
\begin{equation*}
\exists a \in \Z_{\geq 0}, \forall (k,j)\in \Z \times \llbracket 1, r \rrbracket,
\left\{\begin{aligned}
&k<-a &\Rightarrow A(k,j)&=&1,\\
&k>a &\Rightarrow A(k,j)&=&0.
\end{aligned}\right.
\end{equation*}
\end{deff}

\noindent When $r=1$, we will shorten $1$-abacus to abacus. We will represent an $r$-abacus by a superposition of $r$ sequences of holes and beads. The lowest sequence corresponds to $j=1$ and the highest to $j=r$. When $A(k,j)=1$ we will draw a bead and a hole otherwise.
We will also add a dashed vertical line. This line is such that, on each sequence, the first bead or hole directly at the right of the line is at position $i=0$.

\begin{ex}
\label{ex abaque}
Here is an example of a $3$-abacus:
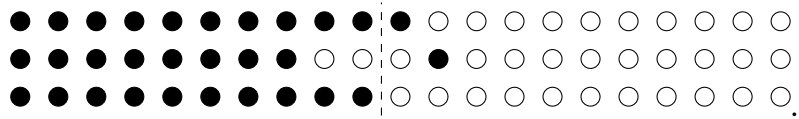
\begin{center}
\begin{tikzpicture}[scale=0.5, bb/.style={draw,circle,fill,minimum size=2.5mm,inner sep=0pt,outer sep=0pt}, wb/.style={draw,circle,fill=white,minimum size=2.5mm,inner sep=0pt,outer sep=0pt}]

	\node [wb] at (11,0) {};
	\node [wb] at (10,0) {};
	\node [wb] at (9,0) {};
	\node [wb] at (8,0) {};
	\node [wb] at (7,0) {};
	\node [wb] at (6,0) {};
	\node [wb] at (5,0) {};
	\node [wb] at (4,0) {};
	\node [wb] at (3,0) {};
	\node [wb] at (2,0) {};
	\node [wb] at (1,0) {};
	\node [bb] at (0,0) {};
	\node [bb] at (-1,0) {};
	\node [bb] at (-2,0) {};
	\node [bb] at (-3,0) {};
	\node [bb] at (-4,0) {};
	\node [bb] at (-5,0) {};
	\node [bb] at (-6,0) {};
	\node [bb] at (-7,0) {};
	\node [bb] at (-8,0) {};
	\node [bb] at (-9,0) {};

	\node [wb] at (11,1) {};
	\node [wb] at (10,1) {};
	\node [wb] at (9,1) {};
	\node [wb] at (8,1) {};
	\node [wb] at (7,1) {};
	\node [wb] at (6,1) {};
	\node [wb] at (5,1) {};
	\node [wb] at (4,1) {};
	\node [wb] at (3,1) {};
	\node [bb] at (2,1) {};
	\node [wb] at (1,1) {};
	\node [wb] at (0,1) {};
	\node [wb] at (-1,1) {};
	\node [bb] at (-2,1) {};
	\node [bb] at (-3,1) {};
	\node [bb] at (-4,1) {};
	\node [bb] at (-5,1) {};
	\node [bb] at (-6,1) {};
	\node [bb] at (-7,1) {};
	\node [bb] at (-8,1) {};
	\node [bb] at (-9,1) {};

	\node [wb] at (11,2) {};
	\node [wb] at (10,2) {};
	\node [wb] at (9,2) {};
	\node [wb] at (8,2) {};
	\node [wb] at (7,2) {};
	\node [wb] at (6,2) {};
	\node [wb] at (5,2) {};
	\node [wb] at (4,2) {};
	\node [wb] at (3,2) {};
	\node [wb] at (2,2) {};
	\node [bb] at (1,2) {};
	\node [bb] at (0,2) {};
	\node [bb] at (-1,2) {};
	\node [bb] at (-2,2) {};
	\node [bb] at (-3,2) {};
	\node [bb] at (-4,2) {};
	\node [bb] at (-5,2) {};
	\node [bb] at (-6,2) {};
	\node [bb] at (-7,2) {};
	\node [bb] at (-8,2) {};
	\node [bb] at (-9,2) {};

\draw[dashed](0.5,-0.5)--node[]{}(0.5,2.5);
\end{tikzpicture}.
\end{center}
\end{ex}

\begin{deff}
Take $\lb \in \mathcal{P}_n$ and $s \in \Z$. We define the abacus $A_{\lb,s}$ associated with $(\lb,s)$ in the following way:
\begin{equation*}
A_{\lb,s}(k)=
\left\{ \begin{aligned}
1 &\quad\text{ if } \exists i \in \mathbb{Z}_{\geq 0}, k=\lb_i -i + s\\
0 &\quad \text{ else}.
\end{aligned} \right.
\end{equation*}
If $\lbb \in \mathcal{P}^r_n$ and $\mathfrak{s} \in \Z^r$, we can associate a $r$-abacus $A_{\lbb,\mathfrak{s}}$ with $(\lbb,\mathfrak{s})$ such that for all $(k,j) \in \mathbb{Z}\times \llbracket 1, r \rrbracket, A_{\lbb,\mathfrak{s}}(k,j):=A_{\lb^j,\mathfrak{s}_j}(k)$.
\end{deff}

\begin{rmq}
Note that from an abacus $A$, there is a unique pair $(\lb,s) \in \mathcal{P}_n \times \Z$ such that $A=A_{\lb,s}$. Indeed, if we number the beads from right to left (i.e. the bead in position $0$ is the rightmost bead), then for each $k \in \mathbb{Z}_{\geq 0}$, $\lb_k$ is equal to the number of holes left to the $k$th bead. If we push all the beads to the left, $s$ is equal to the position of the first hole.
\end{rmq}

\noindent The previous remark easily generalizes to the $r$-abacus. For instance, the $3$-abacus in Example \ref{ex abaque} is the one of the multipartition $(\emptyset,(3),\emptyset)$ with charge $(0,-2,1)$.

We will now recall from \cite[§$4.1$]{JL19} the definition of an $\ell$-elementary operation in an $r$-abacus.

\begin{deff}
Let $A$ be an $r$-abacus. Take $(i,j) \in \Z \times \llbracket 1,r\rrbracket$. An $\ell$-elementary operation at position $(i,j)$ in $A$ is possible
\begin{itemize}
\item If $j<r$, then $A(i,j)=1$ and $A(i,j+1)=0$, or
\item If $j=r$, then $A(i,j)=1$ and $A(i-\ell,1)=0$.
\end{itemize}
If an $\ell$-elementary operation at position $(i,j)$ in $A$ is possible, we obtain a new $r$-abacus $\tilde{A}$ which is such that
\begin{itemize}
\item If $j<r$,
$\left\{\begin{aligned}
&\tilde{A}(i,j)&=&\quad 0,\\
&\tilde{A}(i,j+1)&=&\quad 1,\\
&\tilde{A}(k,t)&=&\quad A(k,t), \quad \forall (k,t) \in \Z \times \llbracket 1,r \rrbracket \setminus \{(i,j),(i,j+1)\}.
\end{aligned}\right.$

\item If $j=r$,
$\left\{\begin{aligned}
&\tilde{A}(i,r)&=&\quad 0,\\
&\tilde{A}(i- \ell,1)&=& \quad 1,\\
&\tilde{A}(k,t)&=& \quad A(k,t), \quad \forall (k,t) \in \Z \times \llbracket 1,r \rrbracket \setminus \{(i,r),(i-\ell,1)\}.
\end{aligned}\right.$
\end{itemize}
\end{deff}

\begin{ex}
If we take again the $3$-abacus of Example \ref{ex abaque}, and we take $\ell=4$, an elementary operation at $(1,2)$ is possible. Doing this operation gives the $3$-abacus
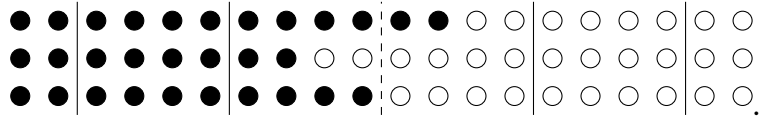
\begin{center}
\begin{tikzpicture}[scale=0.5, bb/.style={draw,circle,fill,minimum size=2.5mm,inner sep=0pt,outer sep=0pt}, wb/.style={draw,circle,fill=white,minimum size=2.5mm,inner sep=0pt,outer sep=0pt}]

	\node [wb] at (10,0) {};
	\node [wb] at (9,0) {};
	\node [wb] at (8,0) {};
	\node [wb] at (7,0) {};
	\node [wb] at (6,0) {};
	\node [wb] at (5,0) {};
	\node [wb] at (4,0) {};
	\node [wb] at (3,0) {};
	\node [wb] at (2,0) {};
	\node [wb] at (1,0) {};
	\node [bb] at (0,0) {};
	\node [bb] at (-1,0) {};
	\node [bb] at (-2,0) {};
	\node [bb] at (-3,0) {};
	\node [bb] at (-4,0) {};
	\node [bb] at (-5,0) {};
	\node [bb] at (-6,0) {};
	\node [bb] at (-7,0) {};
	\node [bb] at (-8,0) {};
	\node [bb] at (-9,0) {};

	\node [wb] at (10,1) {};
	\node [wb] at (9,1) {};
	\node [wb] at (8,1) {};
	\node [wb] at (7,1) {};
	\node [wb] at (6,1) {};
	\node [wb] at (5,1) {};
	\node [wb] at (4,1) {};
	\node [wb] at (3,1) {};
	\node [wb] at (2,1) {};
	\node [wb] at (1,1) {};
	\node [wb] at (0,1) {};
	\node [wb] at (-1,1) {};
	\node [bb] at (-2,1) {};
	\node [bb] at (-3,1) {};
	\node [bb] at (-4,1) {};
	\node [bb] at (-5,1) {};
	\node [bb] at (-6,1) {};
	\node [bb] at (-7,1) {};
	\node [bb] at (-8,1) {};
	\node [bb] at (-9,1) {};

	\node [wb] at (10,2) {};
	\node [wb] at (9,2) {};
	\node [wb] at (8,2) {};
	\node [wb] at (7,2) {};
	\node [wb] at (6,2) {};
	\node [wb] at (5,2) {};
	\node [wb] at (4,2) {};
	\node [wb] at (3,2) {};
	\node [bb] at (2,2) {};
	\node [bb] at (1,2) {};
	\node [bb] at (0,2) {};
	\node [bb] at (-1,2) {};
	\node [bb] at (-2,2) {};
	\node [bb] at (-3,2) {};
	\node [bb] at (-4,2) {};
	\node [bb] at (-5,2) {};
	\node [bb] at (-6,2) {};
	\node [bb] at (-7,2) {};
	\node [bb] at (-8,2) {};
	\node [bb] at (-9,2) {};

\draw[dashed](0.5,-0.5)--node[]{}(0.5,2.5);
\draw[](4.5,-0.5)--node[]{}(4.5,2.5);
\draw[](8.5,-0.5)--node[]{}(8.5,2.5);
\draw[](-3.5,-0.5)--node[]{}(-3.5,2.5);
\draw[](-7.5,-0.5)--node[]{}(-7.5,2.5);
\end{tikzpicture}.
\end{center}
For the sake of readability, we will add vertical lines every $\ell$ positions, when doing multiple $\ell$-elementary operations on an $r$-abacus.
\end{ex}

\section{Blocks of Ariki-Koike algebras}

\begin{deff}
Let $H_{n,r}(\ell,\ssf)$ be the unital, associative $\C$-algebra with presentation\\
\begin{itemize}
\item generators: $T_0,...,T_{n\smin 1}$.
\item relations:
$\left\{
  \begin{aligned}
(T_i+\zeta_{\ell})(T_i-1)=0& &\forall i\in \llbracket 1, n\smin 1\rrbracket,&\\
(T_0-{\zeta_{\ell}}^{\ssf_1})...(T_0-{\zeta_{\ell}}^{\ssf_r})=0&, &&\\
T_iT_j=T_jT_i& &\quad|i\smin j|>1, \forall (i,j)\in \llbracket 0, n \smin 1 \rrbracket^2,&\\
T_iT_{i+1}T_i=T_{i+1}T_iT_{i+1}&  &\forall i \in \llbracket 1, n\smin 2 \rrbracket,&\\
T_0T_1T_0T_1=T_1T_0T_1T_0&.&&
\end{aligned}\right.$
\end{itemize}
The algebra $H_{n,r}(\ell, \ssf)$ is the specialisation at the $\ell$-root of unity $\zeta_{\ell}$ of the Ariki-Koike algebra over $\C$ \cite[Definition $3.1$]{AK94} with cyclotomic parameters $\ssf$ which is the Hecke algebra for the complex reflection group $G(r,1,n)$.
\end{deff}

\noindent We will be interested in the block theory of the algebra $H_{n,r}(\ell,\ssf)$ and its relation with the irreducible components of the fixed points locus in the Gieseker space. To be more precise, take $\mathcal{B}\left(H_{n,r}(\ell,\ssf)\right)$ a complete set of central primitive idempotents of $H_{n,r}(\ell,\ssf)$. To each  $e \in \mathcal{B}\left(H_{n,r}(\ell,\ssf)\right)$, one can construct an indecomposable two-sided ideal $H_{n,r}(\ell,\ssf)e$ which is called a block. This block decomposition gives a partition of the set of irreducible modules of $H_{n,r}(\ell,\ssf)$. We will abuse notation and say that an irreducible module will thus be in a block $B$ if it lies in the subset associated with $B$. For each multipartition $\lbb \in \mathcal{P}^r_n$, there exists a cell $H_{n,r}(\ell,\ssf)$-module $S^{\lbb}$. These modules are called the Specht modules. For $\lbb, \mu^{\bullet} \in \mathcal{P}^r_n$, we will say that $S^{\lbb}$  lies in the same block as $S^{\mu^{\bullet}}$ if there exists a composition factor $D_1$ of $S^{\lbb}$ and a composition factor $D_2$ of $S^{\mu^{\bullet}}$ such that $D_1$ and $D_2$ are in the same block. The main results \cite[Theorem $A$ and $B$]{LM07} show that $S^{\lbb}$ lies in the same block as $S^{\mu^{\bullet}}$ if and only if $\Res_{\ell}(\lbb,\ssf)=\Res_{\ell}(\mu^{\bullet},\ssf)$. The study of blocks from the viewpoint of Specht modules therefore reduces to the combinatorial study of $\mathcal{P}^r_n$ modulo the equivalence relation $\lbb\sim \mu^{\bullet} \iff \Res_{\ell}(\lbb,\ssf)=\Res_{\ell}(\mu^{\bullet},\ssf)$.
We will now recall two block invariants introduced by Fayers. The first one is the hub.

\begin{deff}
For $i \in \Z/\ell\Z$, let $\Theta_{\ell}(\lbb,\ssf)_i$ be the number of removable $i$-nodes of $(\lbb,\ssf)$ minus the number of addable $i$-nodes of $\lbb,\ssf)$. The hub is then the map
\begin{center}
$\Theta_{\ell}: \begin{array}{ccc}
\CP^r_n(\ell) &\to& P_{\Gamma}\\
(\lbb,\ssf) & \mapsto & \sum_{i \in \Z/\ell\Z}{\Theta_{\ell}(\lbb,\ssf)_i\Lambda_i}
\end{array}$.
\end{center}
\end{deff}

\noindent  Let $c_{\Gamma}\in \Hom(Q_{\Gamma}\otimes \C,P_{\Gamma}\otimes \C)$ be the morphism defined such that
\begin{equation*}
c_{\Gamma}(\alpha_i):=\sum_{k \in \Z/\ell\Z}{\langle \alpha_i,\alpha_k\rangle \Lambda_k}, \qquad \forall i\in \Z/\ell\Z.
\end{equation*}

\begin{lemme}
\label{lemme hub res}
For each $(\lbb,\ssf) \in \CP^r_n(\ell)$, $\Theta_{\ell}(\lbb,\ssf) \equiv c_{\Gamma}\left(\Res_{\ell}(\lbb,\ssf)\right) - \Lambda^{\ssf} [\delta^{\Gamma}]$.
\end{lemme}
\begin{proof}
This is a direct reformulation of \cite[Lemma 2.3]{Fay06}.
\end{proof}

\begin{rmq}
This Lemma proves in particular that the hub of an $\ell$-charged, $r$-multipartition is a block invariant.
\end{rmq}

\noindent The second block invariant is the weight, defined in \cite[Section $2.1$]{Fay06}.

\begin{deff}
The weight is the map
\begin{equation*}
\Omega_{\ell}: \begin{array}{ccc}
\CP^r_n(\ell) & \to & \Z_{\geq 0}\\
(\lbb,\ssf)& \mapsto & \sum_{j=1}^r{d_{\ssf_j}}- \frac{1}{2}\sum_{i \in \Z/\ell\Z}{(d_i-d_{i+1})^2} \\
\end{array}
\end{equation*}
where $d=\Res_{\ell}(\lbb,\ssf)$.
\end{deff}

\noindent The hub and weight play an important role in the block theory of Ariki-Koike algebras. Fayers gives an important characterisation \cite[Proposition $1.3$]{Fay07}. We recall it here.
\begin{proposition1}[\cite{Fay07}]
Take two $\ell$-charged, $r$-multipartitions $(\lbb,\ssf)$ and $(\mu^{\bullet},\ssf)$. The following are equivalent
\begin{itemize}
\item $(\lbb,\ssf)$ and $(\mu^{\bullet},\ssf)$ are in the same block.
\item $\Theta_{\ell}(\lbb,\ssf)=\Theta_{\ell}(\mu^{\bullet},\ssf)$ and $|\lbb|=|\mu^{\bullet}|$.
\item $\Theta_{\ell}(\lbb,\ssf)=\Theta_{\ell}(\mu^{\bullet},\ssf)$ and $\Omega_{\ell}(\lbb,\ssf)=\Omega_{\ell}(\mu^{\bullet},\ssf)$.
\end{itemize}
\end{proposition1}

\section{Link with the irreducible components of $\Gie(n,r)^{\Gamma_{\ssf}}$}

In this section, we prove the central results of this article. A first proposition establishes a relation between the dimensions of the irreducible components of $\Gie(n,r)^{\Gamma_{\ssf}}$ and the weight of the blocks.

\begin{prop}
\label{weight dim}
For each $(\lbb,\ssf)\in \CP^r_n(\ell)$, the following equality holds:
\begin{equation*}
\mathrm{dim}\left(\mathcal{M}_{\Res_{\ell}(\lbb, \ssf)}^{\Gamma_{\ssf}}\right)=2\Omega_{\ell}(\lbb,\ssf).
\end{equation*}
\end{prop}
\begin{proof}
If $d=\Res_{\ell}(\lbb,\ssf)$, then
\begin{align*}
\mathrm{dim}\left(\mathcal{M}_{d}^{\Gamma_{\ssf}}\right)&= 2d.w^{\ssf}-\langle d,d\rangle\\
          &=2\left(\sum_{i \in \Z/\ell\Z}{d_iw^{\ssf}_i}-\sum_{i\in \Z/\ell\Z}{d_i^2-d_id_{i+1}}\right).
\end{align*}
Recall that $w^{\ssf}_i=\#\{j \in \llbracket 1, r\rrbracket | \ssf_j=i \}, \qquad \forall i \in \Z/\ell\Z$.\\
On the other side $\Omega_{\ell}(\lbb,\ssf)=\sum_{j=1}^r{d_{\ssf_j}}- \frac{1}{2}\sum_{i \in \Z/\ell\Z}{(d_i-d_{i+1})^2}$, which gives the desired equality.
\end{proof}

\noindent Let $\mathcal{T}$ denote the product of the maximal diagonal torus of $\GL_2(\C)$ and of the maximal diagonal torus of $\GL_r(\C)$. In \cite[section $3.1$]{NakYo03}, the authors proved that the $\mathcal{T}$-fixed points of $\Gie(n,r)$ are parametrized by $r$-multipartitions of size $n$. In \cite[section $2.2$]{Korb13}, the reader can find an explicit construction of this parametrization. Denote by $I_{\lbb}$ the element of $\Gie(n,r)^{\mathcal{T}}$ indexed by the $r$-multipartition $\lbb$ of size $n$. Since $\Gamma_{\ssf} \subset \mathcal{T}$, the fixed points in $\Gie(n,r)$ under $\mathcal{T}$ are in particular fixed under $\Gamma_{\ssf}$. The following proposition is the first connection between $\Gie(n,r)^{\mathcal{T}}$ and the block theory of $H_{n,r}(\ell,\ssf)$.

\begin{prop}
Two points $(I_{\lbb},I_{\mub}) \in \left(\Gie(n,r)^{\mathcal{T}}\right)^2$ are in the same $\Gamma_{\ssf}$-irreducible component if and only if $\lbb$ and $\mub$ are in the same block of $H_{n,r}(\ell,\ssf)$.
\end{prop}
\begin{proof}
By construction of the parametrisation we have that $(I_{\lbb},I_{\mub}) \in \left(\Gie(n,r)^{\mathcal{T}}\right)^2$ are in the same $\Gamma_{\ssf}$-irreducible component if and only if $\Res_{\ell}(\lbb,\ssf)=\Res_{\ell}(\lbb,\ssf)$.
\end{proof}

\noindent Moreover, blocks can also be characterised using the combinatorics of the irreducible components. For $(\lbb,\ssf) \in \CP^r_n(\ell)$, define
\begin{center}
$\left\{
\begin{aligned}
k_{\lbb,\ssf}&:=k_{\Res_{\ell}(\lbb,\ssf)} \in \Z_{\geq 0},\\
\Lambda^+_{\lbb,\ssf}&:=\Lambda^+_{\Res_{\ell}(\lbb,\ssf)} \in P^{++}(\Lambda^{\ssf}),\\
\alpha_{\lbb,\ssf}&:=\Lambda^{\ssf}-\Lambda^+_{\Res_{\ell}(\lbb,\ssf)} \in Q_{\Gamma}.
\end{aligned} \right.$
\end{center}
\noindent By construction, it is clear that these are block invariants. For a block $B \in \mathcal{B}\left(H_{n,r}(\ell,\ssf)\right)$, denote them by $(k_B,\Lambda^+_B)$.
\begin{prop}
Two $\ell$-charged multipartitions $(\lbb,\ssf)$ and $(\mu^{\bullet},\ssf)$ are in the same block if and only if $|\lbb|=|\mu^{\bullet}|$ and $(k_{\lbb,\ssf},\Lambda^+_{\lbb,\ssf})=(k_{\mu^{\bullet},\ssf},\Lambda^+_{\mu^{\bullet},\ssf})$.
\end{prop}
\begin{proof}
The only if part is clear since if $\Res_{\ell}(\lbb,\ssf)=\Res_{\ell}(\mub,\ssf)$, then by definition we have $|\lbb|=|\mu^{\bullet}|$ and $(k_{\lbb,\ssf},\Lambda^+_{\lbb,\ssf})=(k_{\mu^{\bullet},\ssf},\Lambda^+_{\mu^{\bullet},\ssf})$.\\
Now for the converse, it is enough to prove that if $|s_i.d|_{\Gamma}=|d|_{\Gamma}$ for any $i \in \Z/\ell\Z$ and $d \in Q_{\Gamma}$, then $s_i.d=d$. Indeed, by construction, we have that $|\lbb|=|\Res_{\ell}(\lbb,\ssf)|_{\Gamma}$. Now, by definition of the $.$ action, $|s_i.d|=|d|$ if and only if $d_{i+1}+d_{i-1}-d_i+w^{\ssf}_i=d_i$ where $\delta_{i}^{}$ is the Kronecker symbol. This then directly gives that $s_i.d=d$.

\end{proof}

\begin{deff}
A partition $\lb$ is an $\ell$-core if it does not contain any hooks of length $\ell$.
A mulipartition $\lbb \in \mathcal{P}^r_n$ is an $\ell$-multicore if $\lb^c$ is an $\ell$-core for each $c \in \llbracket 1,r \rrbracket$.
\end{deff}

\noindent To finish this subsection, let us recall the generalisation of $\ell$-cores to $r$-multipartitions done by N. Jacon and C. Lecouvey.

\begin{deff}
An $\ell$-movement on an $r$-abacus $A$ at position $(i,j) \in \Z \times \llbracket 1, r \rrbracket$, is a sequence of $r$ $\ell$-elementary operations on $A$ such that $A(i,j)=1$ and $A(i-\ell,j)=0$ and after the $\ell$-movement the $r$-abacus $\tilde{A}$ is equal to $A$ except that $\tilde{A}(i,j)=0$ and $\tilde{A}(i-\ell,j)=1$.
\end{deff}

\begin{rmq}
Note that in terms of charged partitions, if an $\ell$-movement on $A_{\lbb,\mathfrak{s}}$ at position $(i,j)$ exists, then doing this sequence of $\ell$-elementary operations is the same as removing an $\ell$-hook to the partition $\lb^j$. Moreover, as described in \cite[Remark $4.3$]{JL19}, when an $r$-abacus $A$ is such that $A(i,j)=1$ and $A(i-\ell,j)=0$ for a pair ${(i,j) \in \Z \times \llbracket 1, r \rrbracket}$, then there exists an $\ell$-movement on $A$ at position $(i,j)$.
\end{rmq}

\begin{deff}
An $r$-cycle on an $r$-abacus $A_{\lbb,\mathfrak{s}}$ is a sequence of $r$ $\ell$-elementary operations on $A$, such that there exists a multipartition $\mu^{\bullet}$, such that the $r$-abacus obtained after the proceeding of the $\ell$-elementary operations is associated with $(\mu^{\bullet}, \mathfrak{s})$.
\end{deff}

\begin{rmq}
An $\ell$-movement is thus an $r$-cycle.
\end{rmq}

\noindent Consider
\begin{center}
$\varphi^r_{\ell}: \begin{array}{ccc}
{(\Z/\ell\Z)}^r &\iso& \llbracket 0, \ell \smin 1\rrbracket^r\\
\ssf & \mapsto & \left(\ssf_j\scalebox{.7}{\%}\ell\right)_{j \in \llbracket 1, r\rrbracket}  \\
\end{array}$\\
\end{center}
where $k\scalebox{.7}{\%}\ell$ denotes the remainder of the euclidiean division of $k$ by $\ell$.

\noindent Let $\pi_{\ell}:\Z \to \Z/\ell\Z$ denote the quotient map, and consider $\pi^r_{\ell}: \Z^r \to {(\Z/\ell\Z)}^r$.

\noindent Take $(\lbb,\ssf) \in \CP^r_n(\ell)$. Denote by $\mathfrak{s}:=\varphi^r_{\ell}(\ssf)$. Let $\sigma_{\ssf}\in \mathfrak{S}_r$ be the permutation such that
\begin{equation*}
\begin{cases}
\mathfrak{s}_{\sigma_{\ssf}(1)} \leq \mathfrak{s}_{\sigma_{\ssf}(2)} \leq \dots \leq \mathfrak{s}_{\sigma_{\ssf}(r)},\\
\mathfrak{s}_{\sigma_{\ssf}(i)}=\mathfrak{s}_{\sigma_{\ssf}(i+1)} \Rightarrow \sigma_{\ssf}(i) < \sigma_{\ssf}(i+1), \quad \forall i \in \llbracket 1, r \rrbracket.
\end{cases}
\end{equation*}
\noindent Let $\sigma_{\ssf}.(\lbb,\mathfrak{s}):=(\sigma_{\ssf}.\lbb,\sigma_{\ssf}.\mathfrak{s})=\left((\lb^{\sigma_{\ssf}(1)},\dots,\lb^{\sigma_{\ssf}(r)}), (\mathfrak{s}_{\sigma_{\ssf}(1)},\dots ,\mathfrak{s}_{\sigma_{\ssf}(r)})\right)$ and $A_{\lbb,\ssf}$ be the $r$-abacus $A_{\sigma_{\ssf}.(\lbb,\mathfrak{s})}$.

\begin{deff}
The $\ell$-core of the $\ell$-charged multipartition $(\lbb,\ssf)$ is the $\ell$-charged multipartition $(\gamma^{\bullet}, \pi^r_{\ell}(\mathfrak{s}'))$ such that $A_{\gamma^{\bullet}, \mathfrak{s}'}$ is the $r$-abacus obtained from $A_{\lbb,\ssf}$ by doing as much $\ell$-elementary operations as possible.
\end{deff}

\begin{rmq}
$ $\vspace{-0.25cm}
\begin{enumerate}[label=(\roman*)]
\item It is proven \cite[Definition 4.2]{JL19} that the $\ell$-core of an $\ell$-charged multipartition does not depend on the order in which the $\ell$-elementary operations are done. This is done using Uglov's map.
\item The central result \cite[Theorem $3.7$]{JL19} shows that the number of $\ell$-elementary operations to go from $\sigma_{\ssf}.A_{\lbb,\varphi^r_{\ell}(\ssf)}$ to $A_{\mu^{\bullet}, \mathfrak{s}'}$ is equal to $\Omega_{\ell}(\lbb,\ssf)$.
\item If the choice of the isomorphism $\varphi^r_{\ell}$ seems arbitrary for the reader, note that this choice is made such that $\ell$-elementary operations are simple to define.
\item This notion of $\ell$-core generalizes the one of $\ell$-cores for partitions. Indeed, if $r=1$ then the multicharge does not play any role and one recovers the notion of $\ell$-core.
\end{enumerate}
\end{rmq}

\noindent Let us end this subsection, by showing that the new $\ell$-multicharge obtained from the $\ell$-core procedure, can be computed from the maximal dominant weight associated with the $\ell$-charged, $r$-multipartition.

\begin{thm}
\label{link multicharge dom w}
For all $(\lbb,\ssf) \in \CP_n^r(\ell)$,
\begin{equation*}
\Lambda^{\ssf'} \equiv \Lambda^+_{\lbb,\ssf} \quad [\delta^{\Gamma}]
\end{equation*}
\noindent where $(\gammab,\ssf')$ is the $\ell$-core of $(\lbb,\ssf)$.
\end{thm}

\noindent To prove this proposition, we will need an important construction due to Uglov \cite[§4.1]{Ug00}. Before exposing the proof, let us first recall this construction and some properties that we need. Denote by $\llbracket 0, \ell -1\rrbracket^r_{\leq}:=\left\{(\mathfrak{s}_1,\dots \mathfrak{s}_r)\in \llbracket 0, \ell -1\rrbracket^r | \mathfrak{s}_1 \leq \dots \leq \mathfrak{s}_r \right\}$.
\begin{deff}
Let Uglov's map be $\tau:\mathcal{P}^r \times \llbracket 0, \ell -1\rrbracket^r_{\leq}  \to  \mathcal{P} \times \Z_{\geq 0}$ such that the abacus $A$ associated with $\tau(\lbb,\mathfrak{s})$ is defined as follows:
\begin{equation*}
A(i,1):= A_{\lbb,\mathfrak{s}}\left(\left\lfloor \frac{i}{r\ell} \right\rfloor \ell +i\scalebox{.7}{\%}\ell, r-\left\lfloor\frac{i \scalebox{.7}{\%}(r\ell)}{\ell} \right\rfloor \right).
\end{equation*}
By construction, the charge of $\tau(\lbb,\mathfrak{s})$ is equal to $S=\sum_{j=1}^r{\mathfrak{s}_j}$.
\end{deff}

\begin{ex}
For example, if $\ell=4$, $r=3$ and $\lbb=(\emptyset,\emptyset,(2))$ with $\mathfrak{s}=(0,1,2)$, then $A_{\lbb,\mathfrak{s}}$ is
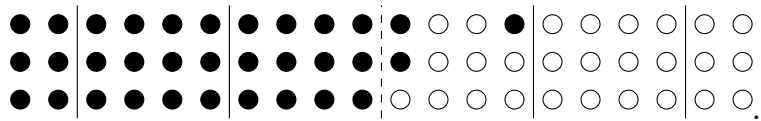
\begin{center}
\begin{tikzpicture}[scale=0.5, bb/.style={draw,circle,fill,minimum size=2.5mm,inner sep=0pt,outer sep=0pt}, wb/.style={draw,circle,fill=white,minimum size=2.5mm,inner sep=0pt,outer sep=0pt}]

	\node [wb] at (10,0) {};
	\node [wb] at (9,0) {};
	\node [wb] at (8,0) {};
	\node [wb] at (7,0) {};
	\node [wb] at (6,0) {};
	\node [wb] at (5,0) {};
	\node [wb] at (4,0) {};
	\node [wb] at (3,0) {};
	\node [wb] at (2,0) {};
	\node [wb] at (1,0) {};
	\node [bb] at (0,0) {};
	\node [bb] at (-1,0) {};
	\node [bb] at (-2,0) {};
	\node [bb] at (-3,0) {};
	\node [bb] at (-4,0) {};
	\node [bb] at (-5,0) {};
	\node [bb] at (-6,0) {};
	\node [bb] at (-7,0) {};
	\node [bb] at (-8,0) {};
	\node [bb] at (-9,0) {};

	\node [wb] at (10,1) {};
	\node [wb] at (9,1) {};
	\node [wb] at (8,1) {};
	\node [wb] at (7,1) {};
	\node [wb] at (6,1) {};
	\node [wb] at (5,1) {};
	\node [wb] at (4,1) {};
	\node [wb] at (3,1) {};
	\node [wb] at (2,1) {};
	\node [bb] at (1,1) {};
	\node [bb] at (0,1) {};
	\node [bb] at (-1,1) {};
	\node [bb] at (-2,1) {};
	\node [bb] at (-3,1) {};
	\node [bb] at (-4,1) {};
	\node [bb] at (-5,1) {};
	\node [bb] at (-6,1) {};
	\node [bb] at (-7,1) {};
	\node [bb] at (-8,1) {};
	\node [bb] at (-9,1) {};

	\node [wb] at (10,2) {};
	\node [wb] at (9,2) {};
	\node [wb] at (8,2) {};
	\node [wb] at (7,2) {};
	\node [wb] at (6,2) {};
	\node [wb] at (5,2) {};
	\node [bb] at (4,2) {};
	\node [wb] at (3,2) {};
	\node [wb] at (2,2) {};
	\node [bb] at (1,2) {};
	\node [bb] at (0,2) {};
	\node [bb] at (-1,2) {};
	\node [bb] at (-2,2) {};
	\node [bb] at (-3,2) {};
	\node [bb] at (-4,2) {};
	\node [bb] at (-5,2) {};
	\node [bb] at (-6,2) {};
	\node [bb] at (-7,2) {};
	\node [bb] at (-8,2) {};
	\node [bb] at (-9,2) {};

\draw[dashed](0.5,-0.5)--node[]{}(0.5,2.5);
\draw[](4.5,-0.5)--node[]{}(4.5,2.5);
\draw[](8.5,-0.5)--node[]{}(8.5,2.5);
\draw[](-3.5,-0.5)--node[]{}(-3.5,2.5);
\draw[](-7.5,-0.5)--node[]{}(-7.5,2.5);
\end{tikzpicture}.
\end{center}

\noindent The abacus of $\tau(\lbb,\mathfrak{s})$ is
\begin{center}
\begin{tikzpicture}[scale=0.5, bb/.style={draw,circle,fill,minimum size=2.5mm,inner sep=0pt,outer sep=0pt}, wb/.style={draw,circle,fill=white,minimum size=2.5mm,inner sep=0pt,outer sep=0pt}]

	\node [wb] at (10,0) {};
	\node [wb] at (9,0) {};
	\node [wb] at (8,0) {};
	\node [wb] at (7,0) {};
	\node [wb] at (6,0) {};
	\node [bb] at (5,0) {};
	\node [bb] at (4,0) {};
	\node [wb] at (3,0) {};
	\node [wb] at (2,0) {};
	\node [bb] at (1,0) {};
	\node [bb] at (0,0) {};
	\node [bb] at (-1,0) {};
	\node [bb] at (-2,0) {};
	\node [bb] at (-3,0) {};
	\node [bb] at (-4,0) {};
	\node [bb] at (-5,0) {};
	\node [bb] at (-6,0) {};
	\node [bb] at (-7,0) {};
	\node [bb] at (-8,0) {};
	\node [bb] at (-9,0) {};

\draw[dashed](0.5,-0.5)--node[]{}(0.5,0.5);
\draw[](4.5,-0.5)--node[]{}(4.5,0.5);
\draw[](8.5,-0.5)--node[]{}(8.5,0.5);
\draw[](-3.5,-0.5)--node[]{}(-3.5,0.5);
\draw[](-7.5,-0.5)--node[]{}(-7.5,0.5);
\end{tikzpicture}.
\end{center}

\noindent Thus $\tau(\lbb,\mathfrak{s})=((2,2),3)$.
\end{ex}

\noindent Let us now present some  properties of Uglov's map

\begin{lemme}
\label{prop 2.22}
Take $(\lbb,\ssf) \in \CP^r_n(\ell)$. If $(\mub,\mathfrak{t}):=\sigma_{\ssf}.(\lbb,\varphi_{\ell}^r(\ssf))$, then
\begin{equation*}
\Res_{\ell}(\lbb,\ssf) \equiv \Res_{\ell}(\tau(\mub,\mathfrak{t}))-\Res_{\ell}(\tau(\emptyset,\mathfrak{t})) \quad[\delta^{\Gamma}].
\end{equation*}
\end{lemme}
\begin{proof}
This is a direct corollary of \cite[Proposition $2.22$]{JL19}.
\end{proof}

\begin{rmq}
Note that the coefficient of $\delta^{\Gamma}$ in Lemma \ref{prop 2.22} is equal to $\langle\Res_{\ell}(\lbb,\ssf),\Lambda_0\rangle r$ thanks to \cite[Proposition $2.22$]{JL19}.
\end{rmq}

\begin{lemme}
\label{lemme somme}
Take $\ssf \in {(\Z/\ell\Z)}^r$. If $\mathcalb{S}:=\sum_{j=1}^r{\ssf_j}$, and $(\mu,\mathfrak{t})=\tau(\sigma_{\ssf}.(\emptyset,\varphi^r_{\ell}(\ssf)))$, then
\begin{equation*}
\Lambda^{\ssf}+\Res_{\ell}(\mu,\mathfrak{t}) \equiv (r-1)\Lambda_0 + \Lambda_{\mathcalb{S}}\quad [\delta^{\Gamma}].
\end{equation*}
\end{lemme}
\begin{proof}
We have that $\mathfrak{t}_1 \leq \mathfrak{t}_2 \leq \dots \leq \mathfrak{t}_r$. For $p \in \llbracket 1, r \rrbracket$, let $\Ssf_0=0$ and $\Ssf_p:=\sum_{j=1}^p{\mathfrak{t}_j}$. Moreoever, for each $i \in \llbracket 1, \Ssf \rrbracket$, define $\beta_{\tkk}(i):=\sum_{j=p+1}^r{(\ell\smin \tkk_j)}$ where $p$ is such that $i \in \llbracket \Ssf_{p\smin 1}+1,\Ssf_p \rrbracket$. We will also need to consider also the Euclidiean division of $\beta_{\tkk}(i)$ by $\ell$. For each $i \in \llbracket 1, \Ssf\rrbracket$, let $\beta_{\tkk}(i)=\kappa_{\tkk}(i)\ell+\rho_{\tkk}(i)$ with $\rho_{\tkk}(i) \in \llbracket 0 ,\ell\smin 1 \rrbracket$. With this notation, the partition $\tau(\emptyset, \tkk)$ is $(\beta_{\tkk}(1),\beta_{\tkk}(2),\dots, \beta_{\tkk}(\Ssf))$. Now:
\begin{equation*}
\Res_{\ell}(\tau(\emptyset,\tkk)) = \sum_{i=1}^{\Ssf}{\left(k_{\tkk}(i)\delta^{\Gamma}+ \sum_{j=1}^{\rho_{\tkk}(i)}{\alpha_{j+\Ssf\smin i}}\right)}.
\end{equation*}
Note that $\beta_{\tkk}(i)=\beta_{\tkk}(i')$ for all $(i,i') \in \llbracket \Ssf_{p\smin 1}+1,\Ssf_p\rrbracket^2$ and that $ \forall i \geq \Ssf_{r\smin 1}, \beta_{\tkk}(i)=0$. Reasoning modulo $\delta^{\Gamma}$ gives:
\begin{equation*}
\Res_{\ell}(\tau(\emptyset,\tkk)) \equiv \sum_{i=1}^{\Ssf_{r\smin 1}}{\sum_{j=1}^{\rho_{\tkk}(i)}{\alpha_{j+\Ssf\smin i}}} \quad [\delta^{\Gamma}].
\end{equation*}
Furthermore:
\begin{center}
\begin{align*}
\sum_{i=1}^{\Ssf_{r\smin 1}}{\sum_{j=1}^{\rho_{\tkk}(i)}{\alpha_{j+\Ssf\smin i}}} &= \sum_{i=1}^{\Ssf_{r\smin 1}}{\sum_{j=1}^{\rho_{\tkk}(i)}{\left(2*\Lambda_{j+ \Ssf \smin i}- \Lambda_{j+\Ssf \smin (i+1)} -\Lambda_{j+ \Ssf \smin (i\smin 1)}\right)}}\\
&= \sum_{p=1}^{r\smin 1}{\sum_{j=1}^{\Ssf_p}{\sum_{i=\Ssf_{p\smin 1}+1}^{\Ssf_p}{\left(2*\Lambda_{j+ \Ssf \smin i}- \Lambda_{j+\Ssf \smin (i+1)} -\Lambda_{j+ \Ssf \smin (i\smin 1)}\right)}}}\\
&= \sum_{p=1}^{r\smin 1}{\sum_{j=1}^{\Ssf_p}{\left(\Lambda_{j+ \Ssf \smin (\Ssf_{p\smin 1}+1)} + \Lambda_{j+\Ssf\smin \Ssf_p} - \Lambda_{j+\Ssf \smin (\Ssf_p +1)} -\Lambda_{j+ \Ssf \smin \Ssf_{p\smin 1}}\right)}}\\
&= \sum_{p=1}^{r\smin 1}{\left(\Lambda_{\Ssf -\Ssf_p + \rho_{\tkk}(\Ssf_p)} - \Lambda_{\Ssf\smin \Ssf_p} + \Lambda_{\Ssf \smin \Ssf_{p\smin 1}}-\Lambda_{\Ssf \smin \Ssf_{p\smin 1} + \rho_{\tkk}(\Ssf_p)}\right)}\\
&= \Lambda_{\Ssf} - \Lambda_{\tkk_{r}} + \sum_{p=1}^{r\smin 1}{\left(\Lambda_{\Ssf \smin \Ssf_p + \rho_{\tkk}(\Ssf_p)}-\Lambda_{\Ssf \smin \Ssf_{p\smin 1} + \rho_{\tkk}(\Ssf_p)}\right)}\\
\end{align*}
\end{center}
Coming back to the definitions, for each $p \in \llbracket1, r\smin 1\rrbracket$, we have:
\begin{center}
$\begin{cases}
\Ssf \smin \Ssf_p + \rho_{\tkk}(\Ssf_p) \equiv 0 \quad[\ell]\\
\Ssf \smin \Ssf_{p\smin 1} + \rho_{\tkk}(\Ssf_p) \equiv \tkk_{p} \quad [\ell]
\end{cases}$
\end{center}
Since $\Lambda^{\ssf}=\sum_{j=1}^r{\Lambda_{\tkk_j}}$. We have proven the desired relation.
\end{proof}

\noindent We have now everything to prove Theorem \ref{link multicharge dom w}.

\begin{proof}[Proof of Theorem \ref{link multicharge dom w}]
Take $(\lbb,\ssf) \in \CP_n^r(\ell)$. Take $\omega \in W_{\Gamma}$, such that
\begin{equation*}
{\omega.\Res_{\ell}(\lbb,\ssf)=\alpha_{\lbb,\ssf} + k_{\lbb,\ssf} \delta^{\Gamma}}.
\end{equation*}

\noindent Let us first show that if $(\mub,\ssf')$ is obtained from $(\lbb,\ssf)$ by doing an $\ell$-elementary operation, then $\omega.\Res_{\ell}(\mub,\ssf')\equiv \alpha_{\mub,\ssf'} [\delta^{\Gamma}]$.
By construction, the partition $\mu:=\tau(\sigma_{\ssf'}.(\mu^{\bullet},\varphi^r_{\ell}(\ssf')))$ is obtained from $\lb=\tau(\sigma_{\ssf}.(\lbb,\varphi^r_{\ell}(\ssf)))$ by removing an $\ell$-hook \cite[Theorem $3.7$]{JL19}. Let $\Ssf:=\sum_{j=1}^r{\ssf_j}$. Lemma \ref{prop 2.22} gives moreover that:
\begin{equation*}
\left\{
\begin{aligned}
\Res_{\ell}(\lbb,\ssf) &\equiv \Res_{\ell}(\lb,\Ssf)-\Res_{\ell}(\tau(\emptyset,\mathfrak{t})) \quad [\delta^{\Gamma}]  \\
\Res_{\ell}(\mub,\ssf') &\equiv \Res_{\ell}(\mu,\Ssf)-\Res_{\ell}(\tau(\emptyset,\mathfrak{t}')) \quad [\delta^{\Gamma}]
\end{aligned}
\right.
\end{equation*}
\noindent where $\mathfrak{t}:=\sigma_{\ssf}.\varphi_{\ell}^r(\ssf)$ and $\mathfrak{t}':=\sigma_{\ssf'}.\varphi_{\ell}^r(\ssf')$.
This implies that
\begin{equation*}
\Res_{\ell}(\lbb,\ssf) \equiv \Res_{\ell}(\mub,\ssf')+\Res_{\ell}(\tau(\emptyset,\mathfrak{t}')) -\Res_{\ell}(\tau(\emptyset,\mathfrak{t})) \quad [\delta^{\Gamma}].
\end{equation*}
Applying $\omega$ gives
\begin{equation*}
\omega*\Lambda^{\ssf} \smin \Lambda^+_{\lbb,\ssf} \equiv \omega*\Res_{\ell}(\mub,\ssf')+\omega*\left(\Res_{\ell}(\tau(\emptyset,\mathfrak{t}'))-\Res_{\ell}(\tau(\emptyset,\mathfrak{t}))\right) \quad [\delta^{\Gamma}].
\end{equation*}
Using Lemma \ref{lemme somme}, we obtain
\begin{equation*}
\omega*\Res_{\ell}(\mub,\ssf') \equiv \omega*\Lambda^{\ssf'} - \Lambda^+_{\lbb,\ssf} \quad [\delta^{\Gamma}].
\end{equation*}

\noindent Using \cite[Proposition $12.6$]{kac90}, we have that
\begin{equation}
\label{link doms}
\Lambda^+_{\mub,\ssf'} \equiv \Lambda^+_{\lbb,\ssf} \quad [\delta^{\Gamma}]
\end{equation}
\noindent Since $\Lambda^+_{\mub,\ssf'}$ is uniquely defined, we can conclude that
\begin{equation*}
\omega.\Res_{\ell}(\mub,\ssf') \equiv \alpha_{\mub,\ssf'} [\delta^{\Gamma}].
\end{equation*}

\noindent Denote by $(\gammab,\ssf')$ the $\ell$-core of $(\lbb,\ssf)$. Thanks to what we have just proven, we have that $\omega.\Res_{\ell}(\gammab,\ssf') = 0$. Combining Proposition \ref{dim easy} with Proposition \ref{weight dim}, gives that ${\alpha_{\gammab,\ssf'}=0}$ and that $k_{\gammab,\ssf'}=0$. Indeed, $(\gammab,\ssf')$ is an $\ell$-core and thus $\Omega_{\ell}(\gammab,\ssf')=0$. Finally, equation (\ref{link doms}) gives that $\Lambda^{\ssf'} \equiv \Lambda^+_{\lbb,\ssf} [\delta^{\Gamma}]$.
\end{proof}

\subsection{Core blocks}
\noindent In \cite[§$3$]{Fay07}, the author introduces the notion of core blocks. Here, we will first recall the central characterisation and then we will give another characterisation that comes from the irreducible components side.

\begin{deff}
A block $B$ of $H_{n,r}(\ell,\ssf)$ is a core block if all multipartitions $\lbb$ that are in $B$ are $\ell$-multicores.
\end{deff}

\noindent The central characterisation of core blocks is \cite[Theorem $3.1$]{Fay07} which we recall.

\begin{deff}
Take $(\lbb,\mathfrak{s}) \in \mathcal{P}^r_n \times \Z^r$ and $(i,j) \in \Z/\ell\Z \times \llbracket 1,r\rrbracket$. Define $\mathfrak{b}^{\mathfrak{s}}_{i,j}(\lbb)$ as the maximum of the following set:
\begin{equation*}
\{k \in \Z | A_{\lbb,\mathfrak{s}}(k,j)=1 \text{ and } \bar{k}=i\}.
\end{equation*}
\end{deff}

\begin{theorem1}[\cite{Fay07}]
Let $B$ a block of $H_{n,r}(\ell,\ssf)$. The following are equivalent.
\begin{itemize}
\item $B$ is a core block.
\item $\forall m \in \Z_{\geq 0}, \forall B' \in \mathcal{B}\left(H_{m,r}(\ell,\ssf)\right), \Theta(B')=\Theta(B) \text{ and } \Omega(B') \leq \Omega(B) \Rightarrow B=B'$.
\item There exists a multipartition $\lbb \in B$ which is an $\ell$-multicore such that there exists $(\mathfrak{s},(a_1,...,a_r))\in ({\Z^r})^2$ verifying the conditions:
\vspace*{0.25cm}

\centerline{\begin{minipage}{0.8\textwidth}
\begin{tasks}{before-skip={0.5cm}}(2)
\task $\pi^r_{\ell}(\mathfrak{s})=\ssf$,
\task $\mathfrak{b}_{ij}^{\mathfrak{s}}(\lbb) \in \{a_i,a_i+\ell\}, \quad \forall (i,j) \in \Z \times \llbracket 1,r\rrbracket$.
\end{tasks}
\end{minipage}}
\end{itemize}
\end{theorem1}

\noindent We are now able to state the central theorem of this article.

\begin{lemme}
\label{link_core_res}
If $(\lbb,\ssf)$ and $(\mu^{\bullet},\ssf)$ are two $\ell$-charged, $r$-multipartitions of the same size, then $\Res_{\ell}(\lbb,\ssf)=\Res_{\ell}(\mu^{\bullet},\ssf)$ if and only if the $\ell$-core of $(\lbb,s)$ is equal to the $\ell$-core of  $(\mu^{\bullet},s)$.
\end{lemme}
\begin{proof}
Combining the result of \cite[Corollary $4.4$]{JL19} with \cite[Theorem $2.11$]{LM07} gives exactly this Lemma.
\end{proof}

\begin{prop}
Take $(\lbb, \ssf) \in \CP^r_n(\ell)$ and denote by $(\gamma_{\ell}^{\bullet},\ttf)$ its $\ell$-core. The number of $r$-cycles that can be done when going from $(\lbb,\ssf)$ to $(\gamma_{\ell}^{\bullet},\ttf)$ is equal to $k_{\Res_{\ell}(\lbb,\ssf)}$ and is in particular a block invariant.
\end{prop}

\begin{thm}
Let $B$ be a block of $H_{n,r}(\ell,\ssf)$. The following are equivalent.
\begin{enumerate}[label=(\alph*)]
\item $B$ is a core block.
\item $B$ contains no $r$-cycles.
\item $k_B=0$.
\end{enumerate}
\end{thm}
\begin{proof}
Let us start by proving that $(a)$ is equivalent to $(b)$. Suppose that we have $(\lbb,\ssf)$ in $B$ that contains an $r$-cyle. Removing this $r$-cycle, we obtain a new $\ell$-charged, $r$-multipartition $(\mub,\ssf)$. An $r$-cyle does not change the charge and thus does not change de $\ell$-charge. Moreover $\Theta_{\ell}(\lbb,\ssf)=\Theta_{\ell}(\mub,\ssf)$ since $\ell$-elementary operations don't change the hub. But we have change the weight and we thus have ${\Omega_{\ell}(\mub,\ssf)=\Omega_{\ell}(\lbb,\ssf)-r}$ which contradicts \cite[Theorem $3.1$]{Fay07}.\\
For any $(i,j)\in \Z/\ell\Z \times \llbracket 1, r\rrbracket$, let us denote by $Z_i^{>j}:=\left\{(k_1,k_2) \in \Z \times \llbracket 1,r \rrbracket \big | \overline{k_1}=i \text{ and } k_2>j \right\}$ and by $Z_i^{\leq j}:=\left\{(k_1,k_2) \in \Z \times \llbracket 1,r \rrbracket \big | \overline{k_1}=i \text{ and } k_2\leq j \right\}$. Suppose now that we have $(\lbb,\ssf)$ in $B$ that does not contain any $r$-cyle. Let $\mathfrak{s}=\varphi^r_{\ell}(\ssf)$. We then have that
\begin{equation*}
\exists j_0 \in \llbracket 1,r\rrbracket, \forall i \in \Z/\ell\Z, \forall (k_1,k_2), (t_1,t_2)) \in Z_i^{>j_0} \times Z_i^{\leq j_0},
\begin{cases}
\sigma_{\ssf}.A_{\lbb,\mathfrak{s}}(k_1,k_2)=1,\\
\text{or }\sigma_{\ssf}.A_{\lbb,\mathfrak{s}}(t_1,t_2)=0,\\
\text{or }k_2 > t_2.
\end{cases}
\end{equation*}
For $i \in \Z/\ell\Z$, define $p_i:=\min\left\{k \in \Z \big | \bar{k}=i \text{ and } \exists j \in \llbracket 1,r \rrbracket, A_{\lbb,\mathfrak{s}}(k,j)=0\right\}$.
This then gives us the following expressions of the $\mathfrak{b}^{\mathfrak{s}}_{i,j}(\lbb)$ For all $(i,j) \in \Z/\ell\Z \times \llbracket 1,j_0 \rrbracket$,
\begin{equation*}
\mathfrak{b}^{\mathfrak{s}}_{i,j}(\lbb)=
\begin{cases}
p_i+i\scalebox{.7}{\%}\ell \text{ or } p_i\smin \ell+i\scalebox{.7}{\%}\ell \quad \text{ if } j \leq j_0\\
p_i+i\scalebox{.7}{\%}\ell \text{ or } p_i+\ell+i\scalebox{.7}{\%}\ell \quad \text{ else}.
\end{cases}
\end{equation*}
This then gives that $B$ is a core block using the first characterisation in \cite[Theorem $3.1$]{Fay07}.\\
There remains to prove that $(b)$ is equivalent to $(c)$. To show this equivalence, it is enough to prove that $(\mub,\ssf)$ is obtained from $(\lbb,\ssf)$ by removing an $r$-cycle if and only if ${\Res_{\ell}(\lbb,\ssf)=\Res_{\ell}(\mub,\ssf)+\delta^{\Gamma}}$. If $(\mub,\ssf)$ is obtained from $(\lbb,\ssf)$ by removing an $r$-cycle, then we have that $\Theta_{\ell}(\lbb,\ssf)=\Theta_{\ell}(\mub,\ssf)$. Furthermore, Lemma \ref{lemme hub res} implies that $c_{\Gamma}(\Res_{\ell}(\lbb,\ssf)-\Res_{\ell}(\mub,\ssf)) \equiv 0 [\delta^{\Gamma}]$. We then have that $t \in \Z$ such that ${\Res_{\ell}(\lbb,\ssf)=\Res_{\ell}(\mub,\ssf) + t\delta^{\Gamma}}$. Using that ${\Omega_{\ell}(\mub,\ssf)=\Omega_{\ell}(\lbb,\ssf)-r}$, we have that ${t=1}$. Conversely, ${\Res_{\ell}(\lbb,\ssf)=\Res_{\ell}(\mub,\ssf)+\delta^{\Gamma}}$ implies that ${\Omega_{\ell}(\lbb,\ssf)=\Omega_{\ell}(\mub,\ssf)+r}$. In particular, it implies that the $r$-multipartition $\mub$ has been obtained from $\lbb$ by doing $r$ $\ell$-elementary operations without changing the $\ell$-charge. We can then conclude that $(\mub,\ssf)$ is obtained from $(\lbb,\ssf)$ by removing an $r$-cycle.
\end{proof}

\begin{rmq}
If $B$ is a block of $H_{n,r}(\ell,\ssf)$ such that $k_B=0$, then each $r$-multipartition $\lbb$ in $B$ is such that there exists $\omega \in W_{\Gamma}$ such that $\Res_{\ell}(\lbb,\ssf)=\omega*\Lambda^+_{B}$. Therefore, $\Res_{\ell}(\lbb,\ssf)$ is a maximal weight of $L(\Lambda^{\ssf})$.
\end{rmq}

\Addresses

\printbibliography
\end{document}